\documentclass[10pt,a4paper]{article}

\usepackage[utf8]{inputenc}
\usepackage[T1]{fontenc}
\usepackage[english]{babel}
\usepackage{amsmath}
\usepackage{amsfonts}
\usepackage{amsthm}
\usepackage{amssymb}
\usepackage{makeidx}
\usepackage{graphicx}
\usepackage{lmodern}
\usepackage[left=2cm,right=2cm,top=2cm,bottom=2cm]{geometry}

\usepackage{color}
\usepackage{comment}
\usepackage[dvipsnames]{xcolor}
\usepackage{graphicx}
\usepackage{framed}
\usepackage{tikz}
\usepackage{calrsfs}
\DeclareMathAlphabet{\pazocal}{OMS}{zplm}{m}{n}

\usepackage{fourier-orns}
\usepackage{mathtools}
\usepackage{relsize}
\usepackage{dsfont}
\usepackage{comment}
\usepackage{hyperref}

\newcommand{\R}{\mathbb{R}}

\renewcommand{\H}{\mathbb{H}}

\newcommand{\U}{\pazocal{U}}
\newcommand{\K}{\pazocal{K}}

\newcommand{\M}{\pazocal{M}}
\newcommand{\Hpazo}{\pazocal{H}}
\newcommand{\Npazo}{\pazocal{N}}
\newcommand{\Lpazo}{\pazocal{L}}
\newcommand{\Ppazo}{\pazocal{P}}

\newcommand{\Cpazo}{\pazocal{C}}

\newcommand{\Lcal}{\mathcal{L}}

\newcommand{\Pcal}{\mathcal{P}}
\newcommand{\Ccal}{\mathcal{C}}

\newcommand{\Ncal}{\mathcal{N}}
\newcommand{\Ucal}{\mathcal{U}}

\newcommand{\Id}{\textnormal{Id}}
\newcommand{\D}{\textnormal{D}}

\newcommand{\Tan}{\textnormal{Tan}}
\newcommand{\Var}{\textnormal{Var}}
\newcommand{\supp}{\textnormal{supp}}
\newcommand{\sign}{\textnormal{sign}}
\newcommand{\Lip}{\textnormal{Lip}}

\newcommand{\Graph}{\textnormal{Graph}}
\newcommand{\loc}{\textnormal{loc}}
\newcommand{\Hess}{\textnormal{Hess} \,}

\newcommand{\xb}{\boldsymbol{x}}
\newcommand{\yb}{\boldsymbol{y}}
\newcommand{\rb}{\boldsymbol{r}}
\newcommand{\sbold}{\boldsymbol{s}}
\newcommand{\pb}{\boldsymbol{p}}
\newcommand{\ub}{\boldsymbol{u}}
\newcommand{\vb}{\boldsymbol{v}}
\newcommand{\wb}{\boldsymbol{w}}
\newcommand{\hb}{\boldsymbol{h}}
\newcommand{\Ab}{\boldsymbol{A}}
\newcommand{\Bb}{\boldsymbol{B}}
\newcommand{\Cb}{\boldsymbol{C}}

\newcommand{\Lb}{\boldsymbol{L}}
\newcommand{\Kb}{\boldsymbol{K}}
\newcommand{\Fb}{\boldsymbol{F}}

\newcommand{\Db}{\textnormal{\textbf{D}}}

\newcommand{\Bphi}{\boldsymbol{\phi}}
\newcommand{\Bvarphi}{\boldsymbol{\varphi}}
\newcommand{\Bnu}{\boldsymbol{\nu}}
\newcommand{\Bsigma}{\boldsymbol{\sigma}}

\newcommand{\INTDom}[3]{\int_{#2} #1 \textnormal{d} #3}
\newcommand{\INTSeg}[4]{\int_{#3}^{#4} #1 \textnormal{d} #2}
\newcommand{\NormL}[3]{\parallel \hspace{-0.1cm} #1 \hspace{-0.1cm} \parallel _ {L^{#2}(#3)}}
\newcommand{\NormC}[3]{\left\| #1  \right\| _ {C^{#2}(#3)}}
\newcommand{\Norm}[1]{\parallel \hspace{-0.1cm} #1 \hspace{-0.1cm} \parallel}
\newcommand{\derv}[2]{\frac{\textnormal{d} #1}{ \textnormal{d} #2}}
\newcommand{\tderv}[2]{\tfrac{\textnormal{d} #1}{ \textnormal{d} #2}}

\newcommand{\BGrad}[1]{\mathbf{Grad}_{#1} \,}
\newcommand{\BHess}[1]{\mathbf{Hess}_{#1} \,}

\newcommand{\textbn}[1]{\textnormal{\textbf{#1}}}
  
\newtheorem{rmk}{Remark}
\newtheorem{lem}{Lemma}
\newtheorem{Def}{Definition}
\newtheorem{thm}{Theorem}
\newtheorem{prop}{Proposition}

\newtheorem{cor}{Corollary}

\newtheorem{taggedhypx}{Hypotheses}
\newenvironment{taggedhyp}[1]
    {\renewcommand\thetaggedhypx{#1}\taggedhypx}
    {\endtaggedhypx}
    
\newtheorem{taggedhypsingx}{Hypothesis}
\newenvironment{taggedhypsing}[1]
    {\renewcommand\thetaggedhypsingx{#1}\taggedhypsingx}
    {\endtaggedhypsingx}

\title{Intrinsic Lipschitz Regularity of Mean-Field Optimal Controls}
\author{Benoît Bonnet\footnote{CNRS,  IMJ-PRG,  UMR  7586,  Sorbonne  Université, 4  place  Jussieu,  75252  Paris,  France. \hfill \textit{E-mail}: \texttt{benoit.bonnet@imj-prg.fr}} , Francesco Rossi\footnote{Dipartimento di Matematica ``Tullio Levi-Civita'', Università degli Studi di Padova, 63 via Trieste, Padova, Italy. \textit{E-mail}: \texttt{francesco.rossi@math.unipd.it}}}

\begin{document}

\maketitle

\begin{abstract}
In this article, we provide sufficient conditions under which the controlled vector fields solution of optimal control problems formulated on continuity equations are Lipschitz regular in space. Our approach involves a novel combination of mean-field approximations for infinite-dimensional multi-agent optimal control problems, along with a careful extension of an existence result of locally optimal Lipschitz feedbacks. The latter is based on the reformulation of a coercivity estimate in the language of Wasserstein calculus, which is used to obtain uniform Lipschitz bounds along sequences of approximations by empirical measures.
\end{abstract}


\medskip

{\small
\textbf{MSC2020 Subject Classification.} 35B65, 49J20, 49J30, 49Q22, 58E25, 93A16
}


\section{Introduction}

The mathematical analysis of collective behaviours in large systems of interacting agents has received an increasing attention from several communities during the past decades. Multi-agent systems are ubiquitous in applications ranging from aggregation phenomena in biology \cite{Camazine2001} to the understanding of crowd motion \cite{CPT}, animal flocks \cite{CS1} and swarms of autonomous vehicles \cite{Bullo2009}. While the first studies on multi-agent systems were formulated in a graph-theoretic framework (see e.g. \cite{Egerstedt2010} and references therein), several models now rely on continuous-time dynamical systems to depict this type of dynamics. In this context, a multi-agent system is usually described by a family of coupled ordinary differential equations (ODEs for short) of the form
\begin{equation}
\label{eq:IntroODE}
\dot x_i(t) = \vb_N(t,\xb(t),x_i(t)), 
\end{equation}
where $\xb = (x_1,\dots,x_N)$ denotes the state of all the agents and $\vb_N : [0,T] \times (\R^d)^N \times \R^d \rightarrow \R^d$ is a \textit{non-local velocity field} depending both on the running agent and on the whole state of the system (see e.g. \cite{bellomo,CPT,CS1}). However general and useful, these models are generically not the most powerful ones when it comes to capturing the global features of a multi-agent system. Besides, their intrinsic dependence on the number $N \geq 1$ of agents makes most of the classical computational approaches practically intractable for realistic scenarios. 

One of the most natural ideas to circumvent this model limitation is to approximate the large system of coupled ODEs written in \eqref{eq:IntroODE} by a single infinite-dimensional dynamics via a process called \textit{mean-field limit} (see e.g. the survey \cite{golse}). In this setting, the agents are supposed to be indistinguishable, and the assembly of particles is described by means of its \textit{spatial density} $\mu(\cdot)$, which is represented by a measure. The evolution through time of this global quantity is then prescribed by a \textit{non-local continuity equation} of the form 
\begin{equation}
\label{eq:IntroPDE}
\partial_t \mu(t) + \nabla \cdot ( v(t,\mu(t),\cdot)\mu(t)) = 0.
\end{equation}
Such a macroscopic approach has been successfully used e.g. to model pedestrian dynamics and biological systems \cite{Camazine2001,CPT}, as well as to transpose the study of classical patterns such as consensus or flocking formation to the mean-field setting \cite{Carrillo2010,HaLiu}. These research endeavours have hugely benefited from the recent progresses made in the theory of \textit{optimal transportation}, for which we point to the reader to the monographs \cite{AGS,OTAM,villani1}. 

More recently, the problem of controlling multi-agent systems in order to promote a desired behaviour or configuration became relevant in a growing number of applications. Motivated by implementability and efficiency, many contributions have therefore aimed at generalising relevant notions of control theory to PDEs of the form \eqref{eq:IntroPDE} serving as mean-field approximations of the discrete system \eqref{eq:IntroODE}. The resulting class of \textit{controlled continuity equations} are usually written as
\begin{equation}
\label{eq:IntroPDEControl}
\partial_t \mu(t) + \nabla \cdot \big( (v(t,\mu(t),\cdot) + u(t,\cdot))\mu(t) \big) = 0.
\end{equation}
While a few articles have been dealing with controllability results \cite{Duprez2019,Duprez2020} or explicit syntheses of control laws \cite{Caponigro2015,ControlKCS}, the major part of the literature has been focusing on \textit{mean-field optimal control problems}, with contributions ranging from existence results \cite{Fornasier2019,FornasierPR2014,Fornasier2014} to first-order optimality conditions \cite{MFPMP,PMPWassConst,SetValuedPMP,PMPWass,Cavagnari2018,Cavagnari2020,Jimenez2020,Pogodaev2016} and numerical methods \cite{achdou2,Burger2020}. 

One of the distinctive features of continuity equations is that they require fairly restrictive regularity assumptions to be classically well-posed. While \eqref{eq:IntroPDEControl} makes sense whenever the drift and control are measurable and satisfy some integrability bounds, the associated notion of so-called \textit{superposition solution} (see e.g. \cite[Theorem 8.2.1]{AGS}) is relatively weak and of limited practical use. In \cite{AmbrosioPDE,DiPernaLions}, a theory of well-posedness was developed for continuity equations with Sobolev and $BV$ velocity fields. However powerful and general, the latter has not yet been generalised to non-local driving fields, and is inherently restricted to measures which are absolutely continuous with respect to the ambient Lebesgue measure. Up to now, the only identified setting in which a strong form of classical well-posedness (see Theorem \ref{thm:Classical_Wellposedness} below) holds for arbitrary measure curves solution of \eqref{eq:IntroPDEControl} is that of \textit{Cauchy-Lipschitz} regularity (see e.g. \cite[Section 3]{AmbrosioC2014}). In this framework, solutions of non-local continuity equations exist, are unique, and stability estimates are available both with respect to the initial data and the velocity fields, see e.g. \cite{BonnetF2021,Pedestrian}. This latter fact is highly relevant to our purpose, since optimal control problems formulated on continuity equations are frequently studied in an ``optimise-then-discretise'' spirit. Indeed, the main desirable property of a control law designed for the kinetic model \eqref{eq:IntroPDEControl} is to provide a strategy which can be in turn applied -- either exactly or approximately -- to finite-dimensional systems of the form \eqref{eq:IntroODE}. As the infinite-dimensional strategy is not strictly optimal for the discrete multi-agent system in general, one would also like to have access to quantitative error estimates between the true solution and the approximate one. From a computational standpoint, Cauchy-Lipschitz regularity is also relevant to ensure the well-posedness of numerical methods such as semi-Lagrangian schemes (see e.g. \cite{Carlini2014,Chang1970}), as well as to prevent the apparition of Lavrentiev-type instabilities in the context of optimal control (see e.g. \cite{Lavrentiev1927}). For all these reasons, a wide portion of the literature  of mean-field control has been dealing with problems in which one imposes an a priori Lipschitz-in-space regularity on the admissible controls (see e.g. \cite{PMPWassConst,PMPWass,SparseJQMF,FornasierPR2014,Fornasier2014,ControlKCS}), or at least some continuity assumptions on the driving fields (see \cite{Cavagnari2018,Cavagnari2020,Jimenez2020}). A natural question is then to ask whether such regularity property may hold intrinsically or not. In this paper, we investigate this problem in the setting of mean-field optimal control problems formulated on the controlled dynamics \eqref{eq:IntroPDEControl}. 

It is well-known that solutions of optimal control problems in Wasserstein spaces need not be regular in general. Indeed, there exists a vast literature devoted to the study of the regularity properties of solutions to the optimal transport problem in Monge formulation (see e.g. \cite{DePhillipis2013,Figalli2013} for some of the farthest-reaching contributions on this topic), mostly via PDE techniques. However, few of these results can be translated into regularity properties on the optimal tangent velocity field $v^*(\cdot,\cdot)$ solving the Benamou-Brenier problem
\begin{equation*}
(\Ppazo_{\textnormal{BB}}) ~~ 
\left\{ 
\begin{aligned}
\min_{v \in L^2} & \left[ \INTSeg{\INTDom{|v(t,x)|^2}{\R^d}{\mu(t)(x)}}{t}{0}{T} \right] \\
\text{s.t.} & \left\{
\begin{aligned}
& \partial_t \mu(t) + \nabla \cdot (v(t,\cdot)\mu(t)) = 0, \\
& \mu(0) = \mu^0 ~~~\text{and}~~~ \mu(T) = \mu^1.
\end{aligned}
\right.
\end{aligned}
\right.
\end{equation*}
For the optimal control problem $(\Ppazo_{\textnormal{BB}})$, it can be shown for instance by building on \cite[Theorem 12.50]{villani1}) that $v(t,\cdot) \in C^{k+1,\alpha}_{\loc}(\R^d,\R^d)$ whenever $\mu^0,\mu^1 \in \Pcal^{ac}(\R^d)$ have densities with respect to the ambient Lebesgue measure with regularity at least $C^{k,\alpha}_{\loc}(\R^d,\R^d)$ for some $\alpha > 0$, and that their supports are smooth and convex. Another context in which the regularity of mean-field optimal controls has been (indirectly) investigated is that of \textit{mean-field games} (see \cite{Huang2006,Lasry2007}). Indeed, there is a large literature dedicated to the regularity of the value function $(t,x) \mapsto V^*(t,x) \in \R$ solving the backward Hamilton-Jacobi equation of the coupled system
\begin{equation*}
\left\{
\begin{aligned}
& \partial_t V(t,x) + H(t,x,\nabla_x V(t,x)) = f(t,x,\mu(t)), \hspace{0.31cm} V(T,x) = g_T(x,\mu(T)), \\
& \partial_t \mu(t) - \nabla \cdot \left( \nabla_p H(t,x,\nabla_x V(t,x)) \mu(t) \right) = 0, \hspace{0.45cm} \mu(0) = \mu^0.
\end{aligned}
\right.
\end{equation*}
In the setting of \textit{variational mean-field games}, the velocity field $v^*(t,x) = -\nabla_p H(t,x,\D_x V^*(t,x))$ is the optimal control associated to a mean-field optimal control problem. Therefore, $v^*(t,\cdot)$ is expected to have a regularity with one order of differentiation fewer than the value function. We refer the reader e.g. to \cite{Cardaliaguet2018} for Sobolev regularity results on the value function $V^*(\cdot,\cdot)$ and to \cite{Cardaliaguet2012} for H\"older regularity (see also Remark \ref{rmk:MFG} below).  

\bigskip

In this article, we investigate the intrinsic Lipschitz regularity with respect to the space variable of the solutions of general mean-field optimal control problems of the form
\begin{equation*}
(\Ppazo) ~~ \left\{
\begin{aligned}
\min_{u(\cdot) \in \U} & \left[ \INTSeg{ \left( L(t,\mu(t)) + \INTDom{\psi(u(t,x))}{\R^d}{\mu(t)(x)} \right)}{t}{0}{T} + \varphi(\mu(T)) \right] \\
\text{s.t.} & \left\{
\begin{aligned}
& \partial_t \mu(t) + \nabla \cdot \big( (v(t,\mu(t),\cdot) + u(t,\cdot))\mu(t) \big) = 0, \\
& \mu(0) = \mu^0.  
\end{aligned}
\right.
\end{aligned}
\right.
\end{equation*}
The set of admissible controls for $(\Ppazo)$ is defined by $\U = L^{\infty}([0,T],L^1(\R^d,U;\mu(t)))$, where $U \subset \R^d$ is a convex and compact set, and $\mu^0 \in \Pcal_c(\R^d)$ is a fixed initial datum. Remark that, since we do not impose any a priori regularity assumptions on the control vector fields $u : [0,T] \times \R^d \rightarrow U$, there may exist no solutions to the non-local transport equation \eqref{eq:IntroPDEControl} driving problem $(\Ppazo)$. Moreover even when solutions do exist, they may not be classically well-posed and defined in a weak sense only. 

The first main contribution of this manuscript is the following existence result of intrinsically Lipschitz-in-space mean-field optimal controls for $(\Ppazo)$.
\begin{thm}[Existence of Lipschitz-in-space solutions for $(\Ppazo)$]
\label{thm:MainResult1}
Let $\mu^0 \in \Pcal_c(\R^d)$ and assume that hypotheses \ref{hyp:H} of Section \ref{section:Fornasier2014} below hold. Moreover, suppose that the control cost $\psi(\cdot)$ is \textnormal{strongly convex} with a constant $\lambda_{\psi} > \lambda_{(\Ppazo)} \geq 0$, where $\lambda_{(\Ppazo)} \geq 0$ only depends on $\supp(\mu^0),T$ and on the $C^2$-norm with respect to the measure and space variables of the dynamics and cost functionals of $(\Ppazo)$.

Then, there exists a constant $\Lpazo_U > 0$ and a trajectory-control pair $(\mu^*(\cdot),u^*(\cdot,\cdot)) \in \Lip([0,T],\Pcal_c(\R^d)) \times \U$ which is optimal for $(\Ppazo)$ such that $x \in \R^d \mapsto u^*(t,x) \in U$ is $\Lpazo_U$-Lipschitz for $\Lcal^1$-almost every $t \in [0,T]$.
\end{thm}

The proof of this result is obtained by combining two fairly separated arguments. The first one is an existence result for mean-field optimal controls which was derived in \cite{Fornasier2019}, and recalled in Theorem \ref{thm:Fornasier2014} below. In the latter, it is proven under very general assumptions that, given a sequence $\mu^0_N := \tfrac{1}{N} \sum_{i=1}^N \delta_{x_i^0}$ converging in the $W_1$-metric towards $\mu^0$, there exist optimal solutions of problem $(\Ppazo)$ which can be recovered as $\Gamma$-limits in a suitable topology of sequences of solutions to the discrete problems
\begin{equation*}
\begin{aligned}
(\Ppazo_N) ~~ \left\{ 
\begin{aligned}
\min_{\ub(\cdot) \in \U_N} & \left[\INTSeg{\Bigg( \Lb_N(t,\xb(t)) + \frac{1}{N}\sum_{i=1}^N \psi(u_i(t)) \Bigg)}{t}{0}{T} + \Bvarphi_N(\xb(T)) \right] \\
\text{s.t.} & \left\{
\begin{aligned}
& \dot x_i(t) = \vb_N(t,\xb_N(t),x_i(t)) + u_i(t), \\
& x_i(0) = x_i^0 \in \R^d.
\end{aligned}
\right.
\end{aligned}
\right.
\end{aligned}
\end{equation*}
Here, $\U_N = L^{\infty}([0,T],U^N)$, and the functionals $\vb_N : [0,T] \times \R^d \times (\R^d)^N \rightarrow \R^d$, $\Lb_N : [0,T] \times (\R^d)^N \rightarrow \R$ and $\Bvarphi_N : (\R^d)^N \rightarrow \R$ are discrete approximations of $v(\cdot,\cdot,\cdot)$, $L(\cdot,\cdot)$ and $\varphi(\cdot)$ respectively (see Definition \ref{def:MF_Adapted} below). 

The second key component of our approach is a careful adaptation of a methodology recently developed in \cite{Cibulka2018,Dontchev2019} to the family of problems $(\Ppazo_N)$, which provides sufficient conditions for the existence of locally optimal Lipschitz feedbacks around solutions of optimal control problems. In the context of mean-field control problems, this part crucially relies on the following \textit{uniform mean-field coercivity estimate} 
\begin{equation*}
\begin{aligned}
\BHess{\xb} \Bvarphi_N[\xb_N^*(T)](\yb(T),\yb(T)) & - \INTSeg{\BHess{\xb} \H_N [t,\xb_N^*(t),\rb_N^*(t),\ub_N^*(t)](\yb(t) , \yb(t))}{t}{0}{T}
\\ 
& - \INTSeg{\BHess{\ub} \H_N [t,\xb_N^*(t),\rb_N^*(t),\ub_N^*(t)](\wb(t) , \wb(t))}{t}{0}{T} \geq \rho_T \INTSeg{|\wb(t)|_N^2}{t}{0}{T}, 
\end{aligned}
\end{equation*}
which holds along any \textit{optimal mean-field Pontryagin triple} $(\ub^*_N(\cdot),\xb^*_N(\cdot),\rb^*_N(\cdot))$ (see Proposition \ref{prop:MFPMP} below) for $(\Ppazo_N)$, where $\H_N : [0,T] \times (\R^{2d})^N \times U \rightarrow \R$ is the Hamiltonian of the discrete problem. In this context, the operator $\BHess{}[\bullet]$ stands for the restriction to empirical measures of the intrinsic Wasserstein Hessian bilinear form (see e.g. \cite{Chow2018}), whose construction is further detailed in Section \ref{section:Preliminary}. In essence, this uniform coercivity estimate allows one to invert the optimality conditions stemming from the application of the \textit{Pontryagin Maximum Principle} (PMP in the sequel) to $(\Ppazo_N)$, with a control on the Lipschitz constant of such inverse. The main subtlety here lies in the fact that we need these estimates to be uniform with respect to $N$, which would not be the case if we were to apply verbatim the results of \cite{Dontchev2019} to $(\Ppazo_N)$. For these reasons, we work with an adapted \textit{mean-field} PMP -- which is the discrete counterpart of the Wasserstein PMP studied in \cite{MFPMP,PMPWassConst,SetValuedPMP,PMPWass} --, and express the coercivity condition in terms of Wasserstein calculus. 

The combination of these two steps together with delicate projection arguments, we can build an optimal feedback $(t,x) \in [0,T] \times \R^d \mapsto u_N^*(t,x) \in U$ for $(\Ppazo_N)$ that is Lipschitz in $x \in \R^d$ uniformly with respect to $N \geq 1$. By standard compactness arguments (see e.g. \cite{MFPMP,Fornasier2014}), this allows us to obtain a result that is stronger than Theorem \ref{thm:MainResult1}, which is the second main contribution of this manuscript. 

\begin{thm}[Convergence of optimal Lipschitz feedbacks towards solutions of $(\Ppazo)$]
\label{thm:MainResult2}
Let $\mu^0 \in \Pcal_c(\R^d)$ and $(\mu^0_N) \subset \Pcal_c(\R^d)$ be a sequence of empirical measures with uniformly compact supports such that $W_1(\mu^0_N,\mu^0) \rightarrow 0$ as $N \rightarrow +\infty$. Suppose that hypotheses \ref{hyp:H} of Section \ref{section:Fornasier2014} below are satisfied, and that the \textnormal{mean-field coercivity estimate} \ref{hyp:CO_N} described in Section \ref{section:MainResult} holds along any optimal Pontryagin triple $(\xb^*_N(\cdot),\rb_N^*(\cdot),\ub^*_N(\cdot))$ for $(\Ppazo_N)$ defined in the sense of Proposition \ref{prop:MFPMP}.

Then, there exists a uniform constant $\Lpazo_U > 0$ depending only on the data of $(\Ppazo)$ and a sequence of trajectory-control pairs $(\mu_N^*(\cdot),u_N^*(\cdot,\cdot)) \subset \Lip([0,T],\Pcal_c(\R^d)) \times L^2([0,T],W^{1,\infty}(\R^d,U))$ such that the following holds.
\begin{enumerate}
\item[\textnormal{(a)}] For any $N \geq 1$ the pair $(\mu^*_N(\cdot),u_N^*(\cdot,\cdot))$ is optimal for $(\Ppazo_N)$, i.e. $\mu^*_N(t) = \tfrac{1}{N}\sum_{i=1}^N \delta_{x_i^*(t)}$ and $u_N^*(t,x_i^*(t)) = u_i^*(t)$ for $\Lcal^1$-almost every $t \in [0,T]$. 
\item[\textnormal{(b)}] The maps $x \in \R^d \mapsto u_N^*(t,x) \in U$ are $\Lpazo_U$-Lipschitz for $\Lcal^1$-almost every $t \in [0,T]$ and any $N \geq 1$. 
\item[\textnormal{(c)}] For every $p \in (d,+\infty)$, the cluster points of the sequence $(u_N^*(\cdot,\cdot))$ in the weak $L^2([0,T],W^{1,p}(\Omega,U))$-topology are optimal controls for $(\Ppazo)$ and are $\Lpazo_U$-Lipschitz in space.
\end{enumerate}
\end{thm}

While they are more general than that of Theorem \ref{thm:MainResult1}, the statements of Theorem \ref{thm:MainResult2} are less intrinsic by nature, as they  rely on the mean-field coercivity estimate \ref{hyp:CO_N} which can only be formulated on the discrete approximations $(\Ppazo_N)$. For this reason, in Proposition \ref{prop:SufficientCoercivity} below, we show that the strong convexity assumption imposed on $\psi(\cdot)$ in Theorem \ref{thm:MainResult1} is in fact a sufficient condition for \ref{hyp:CO_N}. Hence, the statements of Theorem \ref{thm:MainResult1} -- which present the advantage of involving quantities which are intrinsic to $(\Ppazo)$ -- can be recovered as a direct corollary of Theorem \ref{thm:MainResult2}. 

\begin{rmk}[Comparison with related contributions in mean-field games]
\label{rmk:MFG}
It was recently brought to our attention that a result related to Theorem \ref{thm:MainResult1} and Theorem \ref{thm:MainResult2} above was derived in \cite{Gangbo2015}. In the latter, the authors show that the value function of a certain class of first-order mean-field games is continuously differentiable with Lipschitz derivative when the data are of class $C^3$ and the time horizon $T > 0$ is sufficiently small. These two requirement are very close to our standing assumptions. Indeed, we posit in hypotheses \ref{hyp:H} below that all our data are $C^{2,1}_{\loc}$, and it is illustrated in Section \ref{section:Discussion} that our uniform coercivity estimate \ref{hyp:CO_N} can be interpreted as a quantitative condition comparing the size of the time horizon $T>0$ relatively to other constants of the problem, and in particular with the semi-convexity constant of the cost functionals. The results of \cite{Gangbo2015} have then been extended in \cite{Mayorga2020} to broader classes of first-order mean-field games systems, and improved in the very recent \cite{Gangbo2020}, in which it is shown that the value function is regular when the small time horizon condition is replaced by the \textnormal{displacement convexity} (see e.g. \cite[Chapter 9]{AGS}) of the Lagrangian. Incidentally for mean-field control problems, this scenario is also contained in our main result Theorem \ref{thm:MainResult1}. Indeed, if the running cost of the problem is displacement convex, the final cost equal to zero, and the non-local dynamics reduced to a linear controlled vector-field, it can be shown that $\lambda_{(\Ppazo)} = 0$ and the controls are Lipschitz-regular in space whenever the control cost $\psi(\cdot)$ is strongly convex with constant $\lambda > 0$.

We also stress that the proof strategies developed in \cite{Gangbo2020,Gangbo2015,Mayorga2020} are fairly close to the one that we independently propose here, as they rely on the application of inverse function mappings to sequences of approximations by empirical measures, with a quantitative control on the Lipschitz constant of the inverse.
\end{rmk}

\smallskip

The structure of this article is the following. In Section \ref{section:Preliminary}, we recall several general prerequisites on measure theory and optimal transport. In particular in Section \ref{subsection:DiscreteMeasures}, we investigate in details the interplay between Wasserstein derivatives of functionals at empirical measures and classical derivatives of the discrete functionals which arguments are the corresponding support points. In Section \ref{section:FiniteDimOC}, we review notions pertaining to finite-dimensional optimal control problems, with a particular emphasis on Lipschitz feedbacks. We proceed by exposing in Section \ref{section:Fornasier2014} concepts dealing with continuity equations and mean-field optimal control problems, and move on to the proofs of our main results Theorem \ref{thm:MainResult1} and Theorem \ref{thm:MainResult2} in Section \ref{section:MainResult}. More precisely, in Section \ref{subsection:MainResult2}, we state the coercivity assumption \ref{hyp:CO_N} and use it to prove Theorem \ref{thm:MainResult2}. We then show in Section \ref{subsection:MainResult1} how the latter together with a standard convexity estimate for $C^2_{\loc}$-regular functions on convex compact sets allows to recover Theorem \ref{thm:MainResult1}. We conclude by providing in Section \ref{section:Discussion} an analytical example in which our coercivity estimate is both necessary and sufficient for the existence of Lipschitz-in-space mean-field optimal controls. 


\section{Preliminaries}
\label{section:Preliminary}

In this section, we introduce results and notations that we will use throughout the article. Section \ref{subsection:MeasureTheory} presents known results of analysis in measure spaces and optimal transport, while Section \ref{subsection:Wasserstein_Diff} deals with first- and second-order differential calculus in Wasserstein spaces. We introduce in Section \ref{subsection:DiscreteMeasures} the notion of \textit{mean-field approximating sequence}, along with a discretised counterpart of the Wasserstein calculus. 


\subsection{Analysis in measure spaces}
\label{subsection:MeasureTheory}

In this section, we introduce some classical notations and results of analysis in measure spaces and optimal transport. For these topics, we refer the reader to \cite{AmbrosioFuscoPallara} and \cite{AGS,OTAM,villani1} respectively.

We denote by $(\M(\R^d,\R^m),\Norm{\cdot}_{TV})$ the Banach space of $m$-dimensional vector-valued finite Radon measures defined on $\R^d$ endowed with the total variation norm, defined for any $\Bnu \in \M(\R^d,\R^m)$ by $\Norm{\Bnu}_{TV} \, := |\Bnu|(\R^d)$. Here, the total variation measure $|\Bnu| \in \M(\R^d,\R_+)$ associated to $\Bnu$ is given on any Borel set $B \subset \R^d$ by  
\begin{equation*}
|\Bnu|(B) = \sup \left\{ \mathsmaller{\sum}\limits_{k=1}^{+ \infty} |\Bnu(B_k)| ~\text{s.t. $B_k$ are disjoint Borel sets and } \mathsmaller{\bigcup}\limits_{k=1}^{+\infty} B_k = B \right\},
\end{equation*}
where $|\Bnu(B)|$ is the norm of the element $\Bnu(B) \in \R^m$. It is known by Riesz's Theorem (see e.g. \cite[Theorem 1.54]{AmbrosioFuscoPallara}) that $\M(\R^d,\R^m)$ can be identified with the topological dual of the Banach space $(C^0_0(\R^d,\R^m),\Norm{\cdot}_{C^0})$, which is the completion of the space $C^0_c(\R^d,\R^m)$ of continuous and compactly supported functions. The latter is endowed with the duality product 
\begin{equation}
\label{eq:DualityBracket_Measures}
\langle \Bnu , \phi \rangle_{C^0(\R^d,\R^m)} := \sum_{k=1}^m \INTDom{\phi_k(x)}{\R^d}{\nu_k(x)}, 
\end{equation}
defined for any $\Bnu \in \M(\R^d,\R^m)$ and $\phi \in C^0_c(\R^d,\R^m)$. Given a positive Borel measure $\nu \in \M(\R^d,\R_+)$ and an element $p \in [1,+\infty]$, the notations $L^p(\Omega,\R^m;\nu)$ and $W^{1,p}(\Omega,\R^m;\nu)$ stand for the spaces of $p$-integrable and Sobolev functions respectively. In the case where $\nu = \Lcal^d$ is the standard $d$-dimensional Lebesgue measure $\Lcal^d$, we simply denote these spaces by $L^p(\Omega,\R^m)$ and $W^{1,p}(\Omega,\R^m)$.

We use the notation $\Pcal(\R^d) \subset \M(\R^d,\R_+)$ for the space of \textit{Borel probability measures}, and given $p \geq 1$, we denote by $\Pcal_p(\R^d)$ the subset of $\Pcal(\R^d)$ of measures having finite $p$-th moment, i.e. 
\begin{equation*}
\Pcal_p(\R^d) = \Big\{ \mu \in \Pcal(\R^d) ~\text{s.t.}~ \mathsmaller{\INTDom{|x|^p}{\R^d}{\mu(x)}} < +\infty \Big\}.
\end{equation*}
We define the \textit{support} of $\Bnu \in \M(\R^d,\R^m)$ as the closed set $\supp(\Bnu) := \{ x \in \R^d ~\text{s.t.}~ |\Bnu(\pazocal{N})| \neq 0 ~$ for any neighbourhood $\pazocal{N}$ of $x \}$, and denote by $\Pcal_c(\R^d)$ the subset of probability measures with compact support.

\begin{Def}[Absolute continuity and Radon-Nikodym derivative]
\label{def:RadonNikodym}
Let $\Omega \subset \R^m$ and $U \subset \R^d$ be two Borel sets. Given a pair of measures $(\Bnu,\mu) \in \M(\Omega,U) \times \M(\Omega,\R_+)$, we say that $\Bnu$ is \textnormal{absolutely continuous} with respect to $\mu$, and write $\Bnu \ll \mu$, provided that $|\Bnu(B)| = 0$ whenever $\mu(B) = 0$ for any Borel set $B \subset \Omega$. Moreover, it holds that $\Bnu \ll \mu$ if and only if there exists a Borel map $u \in L^1(\Omega,U;\mu)$ such that $\Bnu = u \mu$. This map is referred to as the \textnormal{Radon-Nikodym derivative} of $\Bnu$ with respect to $\mu$, and denoted by $u := \derv{\Bnu}{\mu}$.
\end{Def}

We now recall the definitions of \textit{pushforward} and \textit{transport plan} for Borel probability measure. 
\begin{Def}[Pushforward of a measure through a Borel map] 
Given a measure $\mu \in \Pcal(\R^d)$ and a Borel map $f : \R^d \rightarrow \R^d$, the \textnormal{pushforward} $f_{\#} \mu$ of $\mu$ through $f$ is the Borel probability measure defined by $f_{\#} \mu (B) := \mu(f^{-1}(B))$ for any Borel set $B \subset \R^d$. 
\end{Def}

\begin{Def}[Transport plans]
Let $\mu,\nu \in \Pcal(\R^d)$. We say that $\gamma \in \Pcal(\R^{2d})$ is a \textnormal{transport plan} between $\mu$ and $\nu$, denoted by $\gamma \in \Gamma(\mu,\nu)$, if $\pi^1_{\#} \gamma = \mu$ and $\pi^2_{\#} \gamma = \nu$, where $\pi^1,\pi^2 : \R^{2d} \rightarrow \R^d$ denote the projection on the first and second component respectively.
\end{Def}

In what follows, we recall the definition and some of the main properties of the so-called \textit{Wasserstein spaces} (see e.g. \cite[Chapter 7]{AGS} or \cite[Chapter 6]{villani1}).

\begin{Def}[Wasserstein spaces] 
Given $p \in [1,+\infty)$ and $\mu,\nu \in \Pcal_p(\R^d)$, the \textnormal{Wasserstein distance of order $p$} between $\mu$ and $\nu$ is defined by
\begin{equation*}
W_p(\mu,\nu) := \inf_{\gamma} \bigg\{ \Big( \INTDom{|x-y|^p}{\R^{2d}}{\gamma(x,y)} \Big)^{1/p} ~~ \text{s.t.} ~ \gamma \in \Gamma(\mu,\nu) \bigg\}.
\end{equation*}
The set of \textnormal{optimal transport plans} realising this optimal value is non-empty and denoted by $\Gamma_o(\mu,\nu)$. The space $(\Pcal_p(\R^d),W_p)$ of probability measures with finite momentum of order $p$ endowed with the $W_p$-metric is called the \textnormal{Wasserstein space} of order $p$.
\end{Def}

\begin{prop}[Elementary properties of the Wasserstein spaces]
\label{prop:Properties_Wp}
For any $p \in [1,+\infty)$, the metric space $(\Pcal_p(\R^d),W_p)$ is complete and separable, and the $W_p$-distance metrises the weak-$^*$ topology induced by \eqref{eq:DualityBracket_Measures}, i.e.
\begin{equation*}
W_p(\mu,\mu_n) \underset{n \rightarrow +\infty}{\longrightarrow} 0 \qquad \text{if and only if} \qquad
\left\{
\begin{aligned}
\mu_n & \underset{n \rightarrow +\infty}{\rightharpoonup^*} \mu, \\
\INTDom{|x|^p}{\R^d}{\mu_n(x)} & \underset{n \rightarrow +\infty}{\longrightarrow} \INTDom{|x|^p}{\R^d}{\mu(x)}.
\end{aligned}
\right.
\end{equation*}
Given two elements $\mu,\nu \in \Pcal(\R^d)$, the Wasserstein distances are ordered, i.e. $W_{p_1}(\mu,\nu) \leq W_{p_2}(\mu,\nu)$ whenever $p_1 \leq p_2$. Moreover when $\mu,\nu \in \Pcal_c(\R^d)$, the following \textnormal{Kantorovich-Rubinstein duality formula} holds 
\begin{equation} \label{eq:Kantorovich_duality}
W_1(\mu,\nu) = \sup_{\phi} \left\{ \INTDom{\phi(x) \,}{\R^d}{(\mu-\nu)(x)} ~\text{s.t.}~ \Lip(\phi(\cdot) \, ;\R^d) \leq 1 ~\right\}.
\end{equation}
where $\Lip(\phi(\cdot) \, ; \Omega)$ denotes the Lipschitz constant of $\phi(\cdot)$ over a subset $\Omega \subset \R^d$.
\end{prop}

We end this introductory paragraph by recalling the concept of \textit{disintegration} for vector-valued measures (see e.g. \cite[Theorem 2.28]{AmbrosioFuscoPallara}). 

\begin{thm}[Disintegration]
\label{thm:Disintegration}
Let $\Omega_1 \subset \R^{m_1}$, $\Omega_2 \subset \R^{m_2}$ and $U \subset \R^d$ be Borel sets. Let $\Bnu \in \M(\Omega_1 \times \Omega_2,U)$ and $\pi^1 : \R^{m_1} \times \R^{m_2} \rightarrow \R^{m_1}$ be the projection on the first factor. Defining the measure $\mu := \pi^1_{\#} |\Bnu| \in \M(\Omega_1,\R_+)$, there exists a $\mu$-almost uniquely determined Borel family of measures $\{ \Bnu_x \}_{x \in \Omega_1} \subset \M(\Omega_2,U)$ such that
\begin{equation}
\label{eq:Disintegration}
\INTDom{f(x,y)}{\Omega_1 \times \Omega_2}{\Bnu(x,y)} =  \INTDom{ \left( \INTDom{f(x,y)}{\Omega_2}{\Bnu_x(y)} \right)}{\Omega_1}{\mu(x)} 
\end{equation}
for any Borel map $f \in L^1(\Omega_1 \times \Omega_2, |\Bnu |)$. This construction is referred to as the \textnormal{disintegration} of $\Bnu$ onto $\mu$, and it is denoted by $\Bnu = \INTDom{\Bnu_x}{\Omega_1}{\mu(x)}$. 
\end{thm}


\subsection{First- and second-order differential calculus over $(\Pcal_2(\R^d),W_2)$}
\label{subsection:Wasserstein_Diff}

In this section, we recall key concepts related to first- and second-order differential calculus in the Wasserstein space $(\Pcal_2(\R^d),W_2)$. We refer the reader to \cite[Chapters 9-11]{AGS} and \cite{Gangbo2018} for an exhaustive treatment of the first-order theory, and borrow the main notions dealing with Wasserstein Hessians from \cite[Section 3]{Chow2018}. 

Throughout this section, we denote by $\phi : \Pcal_2(\R^d) \rightarrow \R \cup \{\pm \infty\}$ an extended real-valued functional with non-empty effective domain $D(\phi) =  \{ \mu \in \Pcal_2(\R^d) ~\text{s.t.}~ \phi(\mu) \neq \pm \infty \}$. We will also  denote by $\phi : \Pcal_c(\R^d) \rightarrow \R$ any such functional such that $\Pcal_c(\R^d) \subset D(\phi)$. In the following definition, we recall the notions of \textit{classical subdifferential} and \textit{superdifferential} for functionals defined over $(\Pcal_2(\R^d),W_2)$.

\begin{Def}[Wasserstein subdifferential and superdifferentials]
\label{def:ClassicalSubdifferentials}
Let $\mu \in D(\phi)$. We say that a map $\xi \in L^2(\R^d,\R^d;\mu)$ belongs to the \textnormal{classical subdifferential} $\partial^- \phi(\mu)$ of $\phi(\cdot)$ at $\mu$ provided that 
\begin{equation*} 
\phi(\nu) - \phi(\mu) \geq \sup\limits_{\gamma \in \Gamma_o(\mu,\nu)} \INTDom{\langle \xi(x) , y - x \rangle}{\R^{2d}}{\gamma(x,y)} + o(W_2(\mu,\nu)),
\end{equation*}
for all $\nu \in \Pcal_2(\R^d)$. Similarly, we say that a map $\xi \in L^2(\R^d,\R^d;\mu)$ belongs to the \textnormal{classical superdifferential} $\partial^+ \phi(\mu)$ of $\phi(\cdot)$ at $\mu$ if $(-\xi) \in \partial^- (-\phi)(\mu)$. 
\end{Def}

Following \cite[Chapter 8]{AGS}, we define the \textit{analytical tangent space} $\Tan_{\mu} \Pcal_2(\R^d)$ at $\mu \in \Pcal_2(\R^d)$ by 
\begin{equation}
\label{eq:Definition_TanP2}
\Tan_{\mu} \Pcal_2(\R^d) := \overline{\nabla C^{\infty}_c(\R^d)}^{L^2(\mu)} = \overline{ \big\{ \nabla \xi ~\text{s.t.}~ \xi \in C^{\infty}_c(\R^d) \big\} }^{L^2(\mu)}. 
\end{equation} 
In the next definition, we recall the notion of \textit{differentiable functional} over $\Pcal_2(\R^d)$. 

\begin{Def}[Differentiable functionals in $(\Pcal_2(\R^d),W_2)$]
\label{def:Wasserstein_Diff}
A functional $\phi : \Pcal_2(\R^d) \mapsto \R$ is said to be \textit{differentiable} at $\mu \in D(\phi)$ if $\partial^- \phi(\mu) \cap \partial^+ \phi(\mu) \neq \emptyset$. In this case, there exists a unique elements $\nabla_{\mu} \phi(\mu) \in \partial^- \phi(\mu) \cap \partial^+ \phi(\mu) \cap \Tan_{\mu} \Pcal_2(\R^d)$, called the \textnormal{Wasserstein gradient} of $\phi(\cdot)$ at $\mu$, which satisfies
\begin{equation}
\label{eq:Wasserstein_Diff}
\phi(\nu) - \phi(\mu) = \INTDom{\langle \nabla_{\mu} \phi(\mu)(x) , y - x \rangle}{\R^{2d}}{\gamma(x,y)} + o(W_2(\mu,\nu)),
\end{equation}
for any $\nu \in \Pcal_2(\R^d)$ and $\gamma \in \Gamma_o(\mu,\nu)$. 
\end{Def}

From the characterisation \eqref{eq:Wasserstein_Diff} of the Wasserstein gradient $\nabla_{\mu} \phi(\mu)$, we can write a chain rule along elements of $\Tan_{\mu} \Pcal_2(\R^d)$ (see e.g. \cite[Proposition 10.3.18]{AGS} or the recent improvement of \cite[Proposition 3.6]{SetValuedPMP}).

\begin{prop}[First-order chain rule]
\label{prop:FirstOrder_chain rule}
Suppose that $\phi(\cdot)$ is differentiable at $\mu \in D(\phi)$. Then for any $\xi \in \Tan_{\mu} \Pcal_2(\R^d)$, the map $s \in \R \mapsto \phi((\Id + s \xi)_{\#} \mu)$ is differentiable at $s=0$ with 
\begin{equation}
\label{eq:FirstOrder_Derivative}
\Lpazo_{\xi} \phi(\mu) := \tderv{}{s} \phi((\Id + s \xi)_{\#} \mu) _{\vert s=0} = \INTDom{\langle \nabla_{\mu} \phi(\mu)(x) , \xi(x) \rangle}{\R^d}{\mu(x)}, 
\end{equation}
where $\Lpazo_{\xi} \phi(\mu)$ denotes the \textnormal{Lie derivative} of $\phi(\cdot)$ at $\mu$ in the direction $\xi \in \Tan_{\mu} \Pcal_2(\R^d)$.
\end{prop}
 
In the sequel, we will also need a notion of second-order derivatives for functionals defined over $\Pcal_2(\R^d)$. 

\begin{Def}[Hessian bilinear form in $(\Pcal_2(\R^d),W_2)$]
\label{def:Wasserstein_Hessian}
Suppose that $\phi(\cdot)$ is differentiable at $\mu \in D(\phi)$ and suppose that for any $\xi \in \nabla C^{\infty}_c(\R^d)$, the map 
\begin{equation*}
\Lpazo_{\xi} \phi : \nu \in \Pcal_2(\R^d) \mapsto \langle \nabla_{\mu} \phi(\nu) , \xi \rangle_{L^2(\nu)} \in \R, 
\end{equation*}
is also differentiable at $\mu$. Then, the \textnormal{partial Wasserstein Hessian} of $\phi(\cdot)$ at $\mu$ is the bilinear form defined by 
\begin{equation}
\label{eq:Wasserstein_Hessian}
\Hess \phi[\mu](\xi_1,\xi_2) := \Lpazo_{\xi_2} \left( \Lpazo_{\xi_1} \phi(\mu) \right) - \Lpazo_{\D \xi_1 \xi_2} \phi(\mu),
\end{equation}
for any $\xi_1,\xi_2 \in \nabla C^{\infty}_c(\R^d)$. Moreover, if there exists a constant $C_{\mu} > 0$ such that 
\begin{equation*}
\Hess \phi[\mu](\xi_1,\xi_2) \leq C_{\mu} \NormL{\xi_1}{2}{\mu} \NormL{\xi_2}{2}{\mu},
\end{equation*}
we denote again by $\Hess \phi[\mu](\cdot,\cdot)$ its extension to $\Tan_{\mu} \Pcal_2(\R^d)$ and say that $\phi(\cdot)$ is \textit{twice differentiable} at $\mu$. 
\end{Def}

In the following proposition, we recollect several statements from \cite[Section 3]{Chow2018} which yield an analytical expression of the Wasserstein Hessian. This also allows to write a second-order differentiation formula for functionals defined over $\Pcal_2(\R^d)$. 

\begin{prop}[Wasserstein Hessian and second-order expansion]
\label{prop:Wasserstein_Hessian}
Suppose that $\phi(\cdot)$ is differentiable at $\mu \in D(\phi)$ in the sense of Definition \ref{def:Wasserstein_Diff}, and that the maps 
\begin{equation*}
y \in \R^d \mapsto \nabla_{\mu} \phi(\mu)(y) \in \R^d \qquad \text{and} \qquad \nu \in \Pcal_2(\R^d) \mapsto \nabla_{\mu} \phi(\nu)(x) \in \R^d,
\end{equation*}
are continuously differentiable at $x \in \R^d$ and $\mu \in D(\phi)$ respectively. Then, $\phi(\cdot)$ is twice differentiable in the sense of Definition \ref{def:Wasserstein_Hessian}, and its Wasserstein Hessian can be written explicitly as
\begin{equation}
\label{eq:Explicit_Hessian}
\begin{aligned}
\Hess \phi[\mu](\xi_1,\xi_2) = \INTDom{\big\langle \D_x \nabla_{\mu} \phi(\mu)(x) \xi_1(x),\xi_2(x) \big\rangle}{\R^d}{\mu(x)} + \INTDom{\big\langle \D^2_{\mu} \phi(\mu)(x,y) \xi_1(x),\xi_2(y) \big\rangle}{\R^{2d}}{\mu(x) \textnormal{d}\mu(y)},
\end{aligned}
\end{equation}
for any $\xi_1,\xi_2 \in \Tan_{\mu} \Pcal_2(\R^d)$. Here, $\D_x \nabla_{\mu} \phi(\mu)(x) \in \R^{d \times d}$ is the Fr\'echet differential of $\nabla_{\mu} \phi(\mu)(\cdot)$ at $x \in \R^d$, while $\D^2_{\mu} \phi(\mu)(x,\cdot) : \R^d \rightarrow \R^{d \times d}$ denotes the matrix-valued map whose columns are the Wasserstein gradients of the components $((\nabla_{\mu} \phi(\mu)(x))_i)_{1 \leq i \leq d}$ in the sense of Definition \ref{def:Wasserstein_Diff}. Moreover, the following identity
\begin{equation}
\label{eq:SecondOrder_Wasserstein}
\tderv{}{s} \Lpazo_{\xi_1} \phi((\Id + s \xi_2)_{\#} \mu)_{\vert s=0} = \Hess \phi[\mu](\xi_1,\xi_2) + \Lpazo_{\D \xi_1 \xi_2} \phi(\mu),
\end{equation}
holds for any $\xi_1,\xi_2 \in \nabla C^{\infty}_c(\R^d)$.
\end{prop}

We finally introduce the notion of $\Cpazo^{2,1}_{\loc}$-\textit{Wasserstein regularity}, which will be used throughout this article. 
\begin{Def}[$\Cpazo^{2,1}_{\loc}$-Wasserstein regularity]
\label{def:C2MF}
A functional $\phi(\cdot)$ is said to be \textnormal{$\Cpazo^{2,1}_{\loc}$-Wasserstein regular} if it is twice differentiable over $(\Pcal(K),W_2)$ for any compact set $K \subset \R^d$, and satisfies
\begin{equation}
\label{eq:C21_locMF}
\begin{aligned}
\Norm{\phi(\cdot)}_{\Cpazo^{2,1}(K)} & := \hspace{-0.1cm} \max_{\mu \in \Pcal(K)} |\phi(\mu)| \hspace{-0.025cm} + \hspace{-0.025cm} \NormC{\nabla_{\mu}\phi(\mu)(\cdot)}{0}{K,\R^d} \hspace{-0.025cm} + \hspace{-0.025cm} \NormC{\D_x \nabla_{\mu} \phi(\mu)(\cdot)}{0}{K,\R^{d \times d}} \hspace{-0.025cm} + \hspace{-0.025cm} \NormC{\D^2_{\mu} \phi(\mu)(\cdot,\cdot)}{0}{K \times K,\R^{d \times d}} \hspace{-0.1cm}\\
& \hspace{0.4cm} + \Lip \big( \D_x \nabla_{\mu} \phi(\cdot)(\cdot); \Pcal(K) \times K \big) + \Lip \big( \D^2_{\mu} \phi(\cdot)(\cdot,\cdot); \Pcal(K) \times K \times K \big) \leq C_K,
\end{aligned}
\end{equation}
where $C_K > 0$ is a constant which only depends $K \subset \R^d$.
\end{Def}


\subsection{Mean-field adapted structures and empirical measures}
\label{subsection:DiscreteMeasures}

In this section, we present several notions dealing with functionals defined over
empirical measures in the spirit of \cite{Fornasier2019}, along with an adapted discrete version of the differential structure described in Section \ref{subsection:Wasserstein_Diff}.

We denote by $\Pcal_N(\R^d) := \{ \tfrac{1}{N} \sum_{i=1}^N \delta_{x_i} ~\text{s.t.}~ (x_1,\dots,x_N) \in (\R^d)^N \}$ the set of \textit{$N$-empirical probability measures} over $\R^d$. For any $N \geq 1$, we denote by $\xb = (x_1,\dots,x_N)$ a given element of $(\R^d)^N$ and by $\mu[\xb] := \tfrac{1}{N} \sum_{i=1}^N \delta_{x_i} \in \Pcal_N(\R^d)$ its associated empirical measure. 

\begin{Def}[Symmetric maps defined over $(\R^d)^N$]
A map $\Bphi : (\R^d)^N \rightarrow \R^m$ is said to be \textnormal{symmetric} if $\Bphi \circ \sigma(\cdot) = \Bphi(\cdot)$ for any $d$-blockwise permutation $\sigma : (\R^d)^N \rightarrow (\R^d)^N$.
\end{Def}

In the following definition, we introduce the notion of \textit{mean-field approximating sequence} for continuous functionals defined over $\Pcal_c(\R^d)$.

\begin{Def}[Mean-field approximating sequence]
\label{def:MF_Adapted}
Given an integer $n \geq 1$ and a set $\Omega \subset \R^n$, we define the \textnormal{mean-field approximating sequence} of a  functional $F \in C^0(\Omega \times \Pcal_c(\R^d),\R^m)$ as the family of symmetric maps $(\Fb_N(\cdot,\cdot)) \subset C^0(\Omega \times (\R^d)^N,\R^m)$, defined by 
\begin{equation}
\label{eq:MF_Map}
\Fb_N(x,\xb) := F(x,\mu[\xb]),
\end{equation}
for any $N \geq 1$ and all $(x,\xb) \in \Omega \times (\R^d)^N$.
\end{Def}

We henceforth endow the vector space $(\R^d)^N$ with the rescaled inner product $\langle \cdot, \cdot \rangle_N$ defined by  
\begin{equation}
\label{eq:Adapted_Products}
\langle \xb , \yb \rangle_N =  \tfrac{1}{N} \mathsmaller{\sum}\limits_{i=1}^N \langle x_i , y_i\rangle,
\end{equation} 
for any $\xb,\yb \in (\R^d)^N$, where $\langle \cdot,\cdot \rangle$ is the standard Euclidean product of $\R^d$. We also denote by $|\cdot|_N = \sqrt{\langle \cdot,\cdot \rangle_N}$ the corresponding norm over $(\R^d)^N$, and observe that $((\R^d)^N,\langle \cdot,\cdot \rangle_N)$ is an Hilbert space. 

In the following proposition, we show that the Wasserstein differential structure described in Section \ref{subsection:Wasserstein_Diff} for functionals defined over $\Pcal_2(\R^d)$ induces a natural differential structure on the Hilbert space $((\R^d)^N,\langle \cdot,\cdot \rangle_N)$. We will use the notation $C^{2,1}_{\loc}$ to refer to functionals between finite-dimensional normed vector spaces which are twice differentiable with locally Lipschitz derivatives up to the second-order. 

\begin{prop}[Mean-field derivatives of symmetric maps]
\label{prop:MF_Derivatives}
Let $\phi(\cdot)$ be $\Cpazo^{2,1}_{\loc}$-Wasserstein regular in the sense of Definition \ref{def:C2MF} above and $(\Bphi_N(\cdot)) \subset C^0((\R^d)^N)$ be the mean-field approximating sequence of $\phi(\cdot)$. 

Then, $\Bphi_N \in C^{2,1}_{\loc}((\R^d)^N,\R)$ for any $N \geq 1$, and the following Taylor expansion formula
\begin{equation}
\label{eq:MF_Taylor}
\Bphi_N(\xb + \hb) = \Bphi_N(\xb) + \langle \BGrad{} \Bphi_N(\xb) , \hb \rangle_N + \tfrac{1}{2} \BHess{} \Bphi_N[\xb](\hb,\hb) + o(|\hb|_N^2),
\end{equation}
holds for any $\xb,\hb \in (\R^d)^N$. Here, we introduced the \textnormal{mean-field gradient} $\BGrad{}\Bphi_N(\cdot)$ and \textnormal{mean-field Hessian bilinear form} $\BHess{} \Bphi_N[\cdot]$ of $\Bphi_N(\cdot)$, given respectively by
\begin{equation}
\label{eq:MF_Gradient}
\BGrad{} \Bphi_N(\xb) := (\nabla_{\mu} \phi(\mu[\xb])(x_i))_{1 \leq i \leq N},
\end{equation}
and 
\begin{equation}
\label{eq:MF_Hessian}
\begin{aligned}
\BHess{} \Bphi_N[\xb](\hb^1,\hb^2) := \frac{1}{N} \sum_{i=1}^N \langle \D_x \nabla_{\mu} \phi(\mu[\xb])(x_i) h_i^1 , h_i^2 \rangle_N + \frac{1}{N^2} \sum_{i,j = 1}^N \langle \D^2_{\mu} \phi(\mu[\xb])(x_i,x_j) h_i^1,h_j^2 \rangle, 
\end{aligned}
\end{equation}
for any $\xb,\hb^1,\hb^2 \in (\R^d)^N$. Defining the $C^{2,1}_N$-norm of $\Bphi_N(\cdot)$ over $\Kb \subset (\R^d)^N$ with respect to differential structure of $((\R^d)^N,\langle \cdot,\cdot \rangle_N)$ as 
\begin{equation}
\label{eq:C2Norm}
\Norm{\Bphi_N(\cdot)}_{C^{2,1}_N(\Kb)} \, := \,  \max_{\xb \in \Kb} \Big( \Bphi_N(\xb) + |\BGrad{} \Bphi_N(\xb)|_N + \max_{|\hb|_N = 1} |\BHess{} \Bphi_N[\xb](\hb,\hb)|\Big) +  \Lip \left( \BHess{} \Bphi_N[\cdot] \, ; \Kb \right), 
\end{equation}
it further holds for each compact set $K \subset \R^d$ that
\begin{equation}
\label{eq:C2_Correspondance}
\Norm{\Bphi_N(\cdot)}_{C^{2,1}_N(K^N)} ~ \leq ~ \Norm{\phi(\cdot)}_{\Cpazo^{2,1}(K)},
\end{equation}
where the $\Cpazo^{2,1}$-Wasserstein norm $\Norm{\phi(\cdot)}_{\Cpazo^{2,1}(K)}$ of $\phi(\cdot)$ is defined as in \eqref{eq:C21_locMF}. 
\end{prop}

\begin{proof}
Take $\xb := (x_1,\dots,x_N) \in (\R^d)^N$ and $\hb := (h_1,\dots,h_N) \in (\R^d)^N$ and define $\epsilon := \tfrac{1}{4} \min_{x_i \neq x_j} |x_i - x_j|$. Consider the map $\zeta_N(\cdot)$ given by
\begin{equation*}
\zeta_N : x \in \R^d \mapsto \left\{
\begin{aligned}
& \langle x , h_i \rangle ~~ & \text{if $x \in B(x_i,2\epsilon)$ with $i \in \{1,\dots,N\}$}, \\
& 0 ~~ &\text{otherwise},
\end{aligned}
\right.
\end{equation*}
and let $\eta \in C^{\infty}_c(\R^d)$ be a symmetric mollifier centred at the origin and supported on $B(0,\epsilon)$. We define the tangent vector $\xi_N \in \nabla C^{\infty}_c(\R^d) \subset \Tan_{\mu[\xb]} \Pcal_2(\R^d)$ at $\mu[\xb]$ by 
\begin{equation}
\label{eq:MF_ProofXi}
\xi_N : x \in \R^d \mapsto \nabla ( \eta * \zeta_N )(x).
\end{equation}
Remark that by construction, one has 
\begin{equation}
\label{eq:MF_Proof1}
\xi_N(x_i) = h_i \qquad \text{and} \qquad \D_x \xi_N(x_i) = 0, 
\end{equation}
so that in particular $\mu[\xb+ s \hb] = (\Id + s \xi_N)_{\#}\mu[\xb]$ for any $s \in \R$ sufficiently small.  

Recall now that $\phi(\cdot)$ is differentiable at $\mu[\xb] \in \Pcal_c(\R^d)$ by hypothesis. Hence by Proposition \ref{prop:FirstOrder_chain rule}, it holds 
\begin{equation*}
\lim_{s \rightarrow 0} \left[ \frac{\phi(\mu[\xb + s\hb]) - \phi(\mu[\xb])}{s} \right] = \Lpazo_{\xi_N} \phi(\mu[\xb]) = \INTDom{ \langle \nabla_{\mu} \phi(\mu[\xb])(x) , \xi_N(x) \rangle}{\R^d}{\mu[\xb](x)}. 
\end{equation*}
Recalling the definition of approximating maps $\Bphi_N(\cdot)$ given in \eqref{eq:MF_Map}, we further obtain 
\begin{equation}
\label{eq:MF_Proof2}
\Bphi_N'(\xb;\hb) := \lim_{s \rightarrow 0} \left[ \frac{\Bphi_N(\xb + s\hb) - \Bphi_N(\xb)}{s} \right] =  \frac{1}{N} \sum_{i=1}^N \langle \nabla_{\mu} \phi(\mu[\xb])(x_i),h_i \rangle, 
\end{equation}
where we used \eqref{eq:MF_Proof1} along with the fact that $\mu[\xb] = \tfrac{1}{N} \sum_{i=1}^N \delta_{x_i}$. It is straightforward to check that the directional derivative $\hb \mapsto \Bphi_N'(\xb;\hb)$ of $\Bphi_N(\cdot)$ defined in \eqref{eq:MF_Proof2} is a linear form and that it is continuous with respect to the rescaled Euclidean metric $|\cdot|_N$. Whence, the map $\Bphi_N(\cdot)$ is Fr\'echet differentiable at $\xb \in (\R^d)^N$, and by Riesz's Theorem (see e.g. \cite[Theorem 5.5]{Brezis}), its differential can be represented in the Hilbert space $((\R^d)^N,\langle \cdot , \cdot \rangle_N)$ by the mean-field gradient $\BGrad{}\Bphi_N(\xb) := (\nabla_{\mu}\phi(\mu[\xb])(x_i))_{1 \leq i \leq N}$ defined in \eqref{eq:MF_Gradient}. 

Consider now two elements $\hb^1,\hb^2 \in (\R^d)^N$ and the corresponding tangent vectors $\xi^1_N,\xi_N^2 \in \nabla C^{\infty}_c(\R^d)$ built as in \eqref{eq:MF_ProofXi}. Since $\phi(\cdot)$ is twice differentiable in the sense of Definition \ref{def:Wasserstein_Hessian}, it holds by \eqref{eq:SecondOrder_Wasserstein} in Proposition \ref{prop:Wasserstein_Hessian} 
\begin{equation}
\label{eq:MF_Proof3}
\lim_{s \rightarrow 0} \left[ \frac{\Lpazo_{\xi_N^1} \phi((\Id + s \xi_N^2)_{\#} \mu[\xb]) - \Lpazo_{\xi_N^1} \phi(\mu[\xb])}{s}\right] \hspace{-0.1cm} = \hspace{-0.075cm} \Hess \phi[\mu[\xb]](\xi_N^1,\xi_N^2) + \Lpazo_{\D \xi_N^1 \xi_N^2} \phi(\mu[\xb]). 
\end{equation}
Observe now that $\D \xi_N^1(x) = 0$ for $\mu[\xb]$-almost every $x \in \R^d$ by \eqref{eq:MF_ProofXi}, so that $\Lpazo_{\D \xi_N^1 \xi_N^2} \phi(\mu[\xb]) = 0$. Furthermore, by the definition of $\Bphi_N(\cdot)$ along with that of $\BGrad{} \Bphi_N(\cdot)$, equation \eqref{eq:MF_Proof3} can be equivalently rewritten as 
\begin{equation}
\label{eq:MF_Proof4}
\begin{aligned}
& \lim_{s \rightarrow 0} \left[ \frac{\langle \BGrad{} \Bphi_N(\xb + s \hb^2) - \BGrad{} \Bphi_N(\xb) , \hb^1 \rangle_N}{s}  \right] \\
& \hspace{2.5cm} = \frac{1}{N} \sum_{i=1}^N \langle \D_x \nabla_{\mu} \phi(\mu[\xb])(x_i) h_i^1 , h_i^2 \rangle + \frac{1}{N^2} \sum_{i,j=1}^N \langle \D^2_{\mu} \phi(\mu[\xb])(x_i,x_j) h_i^1 , h_j^2 \rangle,
\end{aligned}
\end{equation}
where we used the analytical expression \eqref{eq:Explicit_Hessian} of the Wasserstein Hessian. We accordingly introduce the mean-field Hessian bilinear form $\BHess{} \Bphi_N[\xb](\cdot,\cdot)$ of $\Bphi_N(\cdot)$ at $\xb \in (\R^d)^N$, defined as in \eqref{eq:MF_Hessian}. It is again possible to verify that $\BHess{} \Bphi_N [\xb](\cdot,\cdot)$ defines a continuous bilinear form with respect to the rescaled metric $|\cdot|_N$, so that the map $\Bphi_N(\cdot)$ is twice Fr\'echet differentiable over $(\R^d)^N$. The expansion formula \eqref{eq:MF_Taylor} can then be derived by developing $\Bphi_N(\xb+\hb)$ using the classical Taylor theorem in $(\R^d)^N$ along with \eqref{eq:MF_Proof2} and \eqref{eq:MF_Proof4}.

We now prove the regularity bound of \eqref{eq:C2_Correspondance}. Given $K \subset \R^d$, we obtain from the fact that $(\Bphi_N(\cdot))$ is a mean-field approximating sequence for $\phi(\cdot)$ together with the definition of $\BGrad{} \Bphi_N(\cdot)$ displayed in \eqref{eq:MF_Gradient}, that
\begin{equation}
\label{eq:C21N_Est1}
\max_{\xb \in K^N} |\Bphi_N(\xb)| = \max_{\xb \in K^N} |\phi(\mu[\xb])| \leq \max_{\mu \in \Pcal(K)} |\phi(\mu)|,
\end{equation}
and 
\begin{equation}
\label{eq:C21N_Est2}
\max_{\xb \in K^N} |\BGrad{} \Bphi_N(\xb)|_N = \max_{\xb \in K^N} \Big( \tfrac{1}{N} \mathsmaller{\sum}_{i=1}^N |\nabla_{\mu} \phi(\mu[\xb])(x_i)|^2 \Big)^{1/2} \leq \max_{\mu \in \Pcal(K)} \NormC{\nabla_{\mu} \phi(\mu)(\cdot)}{0}{K,\R^d}. 
\end{equation}
Analogously, using the definition of $\BHess{} \Bphi_N(\cdot)$ given in \eqref{eq:MF_Hessian}, we can deduce 
\begin{equation}
\label{eq:C21N_Est3}
\max_{\xb \in K^N} \BHess{} \Bphi[\xb](\hb,\hb) ~ \leq ~ \max_{\mu \in \Pcal(K)} \Big( \NormC{\D_x \nabla_{\mu} \phi(\mu)(\cdot)}{0}{K,\R^{d \times d}} + \NormC{\D_{\mu}^2 \phi(\mu)(\cdot,\cdot)}{0}{K \times K,\R^{d \times d}} \Big),
\end{equation}
as well as the Lipschitz estimate
\begin{equation}
\label{eq:C21N_Est4}
\Lip(\BHess{} \Bphi_N[\cdot] \, ; K^N) \leq \Lip(\D_x \nabla_{\mu} \phi(\cdot)(\cdot) \, ; \Pcal(K) \times K) + \Lip(\D_{\mu}^2 \phi(\cdot)(\cdot,\cdot) \, ; \Pcal(K) \times K \times K),
\end{equation}
where we used the fact that $W_2(\mu[\xb],\mu[\yb]) \leq |\xb-\yb|_N$ for $\xb,\yb \in (\R^d)^N$. By plugging \eqref{eq:C21N_Est1}, \eqref{eq:C21N_Est2}, \eqref{eq:C21N_Est3} and \eqref{eq:C21N_Est4} into \eqref{eq:C2Norm} and recalling the definition \eqref{eq:C21_locMF} of $\Norm{\phi(\cdot)}_{\Cpazo^{2,1}(K)}$, we conclude that \eqref{eq:C2_Correspondance} holds.
\end{proof}

\begin{rmk}[Matrix representation of the mean-field Hessian in $(\R^d)^N$]
\label{rmk:MF_Hessian}
By Riesz's Theorem applied in the Hilbert space $((\R^d)^N),\langle \cdot,\cdot \rangle_N)$, the action of the Hessian bilinear form $\BHess{} \Bphi_N[\xb](\cdot,\cdot)$ can be represented as
\begin{equation}
\label{eq:MF_HessianMatrix}
\BHess{} \Bphi_N[\xb](\hb^1,\hb^2) = \big\langle \BHess{} \Bphi_N(\xb)  \hb^1 , \hb^2 \big\rangle_N, 
\end{equation}
for any $\xb,\hb^1,\hb^2 \in (\R^d)^N$, where $\BHess{} \Bphi_N(\xb) \in \R^{dN \times dN}$ is a matrix. In this case, its components can be obtained via a simple identification in \eqref{eq:MF_Hessian}, and be written explicitly as
\begin{equation*}
(\BHess{} \Bphi_N(\xb))_{i,j} = \D^2_{\mu} \phi(\mu[\xb])(x_i,x_j), \qquad (\BHess{} \Bphi_N(\xb))_{i,i} = N \D_x \nabla_{\mu} \phi(\mu[\xb])(x_i) + \D^2_{\mu} \phi(\mu[\xb])(x_i,x_i), 
\end{equation*}
for any pair of indices $i,j \in \{ 1,\dots,N\}$ such that $i \neq j$. 
\end{rmk}


\section{Locally optimal Lipschitz feedbacks in optimal control}
\label{section:FiniteDimOC}

In this section, we recall classical facts about finite dimensional optimal control problems, and describe in Theorem \ref{thm:LipFeedback} a result proven in \cite{Dontchev2019}, which provides sufficient conditions for the existence of locally optimal Lipschitz feedbacks in a neighbourhood of an optimal trajectory. Throughout this section, we will study the finite-dimensional optimal control problem
\begin{equation*}
(\Ppazo_{\textnormal{oc}}) ~~ \left\{ 
\begin{aligned}
\min_{u(\cdot) \in \U} & \left[ \INTSeg{ \Big( l(t,x(t)) + \psi(u(t)) \Big)}{t}{0}{T} + g(x(T)) \right] \\
\text{s.t.} ~ & \left\{
\begin{aligned}
\dot x(t) & = f(t,x(t)) + u(t), \\
x(0) & = x^0, \\
\end{aligned}
\right.
\end{aligned}
\right.
\end{equation*}
under the following assumptions. 

\begin{taggedhyp}{\textbn{(H$_{\text{oc}}$)}} \hfill
\label{hyp:Hoc}
\begin{enumerate}
\item[(i)] The set of admissible controls is given by $\U = L^{\infty}([0,T],U)$ where $U \subset \R^d$ is convex and compact. 
\item[(ii)] The control cost $u \mapsto \psi(u) \in \R$ is $C^{2,1}$-regular and strictly convex over $U$. 
\item[(iii)] The map $(t,x) \mapsto f(t,x) \in \R^d$ is Lipschitz with respect to $t \in [0,T]$ and $C^{2,1}_{\loc}$-regular with respect to $x \in \R^d$. Moreover, there exists a constant $M > 0$ such that 
\begin{equation*}
|f(t,x)| \leq M(1 + |x|), 
\end{equation*}
for any $(t,x) \in [0,T] \times \R^d$. 
\item[(iv)] The running cost $(t,x) \mapsto l(t,x) \in \R$ is Lipschitz with respect to $t \in [0,T]$ and $C^{2,1}_{\loc}$-regular with respect to $x \in \R^d$. Similarly, the final cost $x \mapsto g(x) \in \R$ is $C^{2,1}_{\loc}$-regular over $\R^d$.
\end{enumerate}
\end{taggedhyp}

It can be easily seen that one could choose integrable maps to express the sub-linearity and Lipschitz regularity of $f(\cdot,\cdot)$ instead of constants. As a direct consequence of \ref{hyp:Hoc}, we have the following lemma. 

\begin{lem}[Uniform compactness of admissible trajectories]
\label{lem:Compactness_FiniteDim}
Given $x^0 \in \R^d$, there exists a compact set $K \subset \R^d$ such that each admissible curve $x(\cdot)$ for $(\Ppazo_{\textnormal{oc}})$ associated to a control $u(\cdot) \in \U$ satisfies $x(\cdot) \in \Lip([0,T],K)$.  
\end{lem}

\begin{proof}
This follows directly from an application of Gr\"onwall's Lemma.  
\end{proof}

\begin{prop}[Existence of solutions for problem $(\Ppazo_{\textnormal{oc}})$]
\label{prop:Existence_FiniteDim}
Let $K \subset \R^d$ be a compact set given as in Lemma \ref{lem:Compactness_FiniteDim} and suppose that hypotheses \ref{hyp:Hoc} hold. Then, there exists an optimal trajectory-control pair $(x^*(\cdot),u^*(\cdot)) \in  \Lip([0,T],K) \times \U$  for problem $(\Ppazo_{\textnormal{oc}})$. 
\end{prop}

\begin{proof}
This result is standard under our working hypotheses and can be found e.g. in \cite[Theorem 23.11]{Clarke}. 
\end{proof}

We introduce the \textit{Hamiltonian} function associated with $(\Ppazo_{\textnormal{oc}})$, defined by 
\begin{equation*}
H : (t,x,p,u) \in [0,T] \times (\R^d)^3 \mapsto \langle p , f(t,x) + u \rangle -  \big( l(t,x) + \psi(u) \big). 
\end{equation*} 
Let $(x^*(\cdot),u^*(\cdot))$ be an optimal trajectory-control pair for $(\Ppazo_{\textnormal{oc}})$. By the \textit{Pontryagin Maximum Principle} (see e.g. \cite[Theorem 22.2]{Clarke}), there exists a curve $p^*(\cdot)$ such that the couple $(x^*(\cdot),p^*(\cdot))$ is a solution of the \textit{forward-backward Hamiltonian system}
\begin{equation}
\label{eq:Hamiltonian_FiniteDim}
\left\{
\begin{aligned}
\dot x^*(t) = \hspace{0.4cm} & \nabla_p H(t,x^*(t),p^*(t),u^*(t)), \quad x^*(0) = x^0, \\
\dot p^*(t) = -& \nabla_x H(t,x^*(t),p^*(t),u^*(t)), \quad p^*(T) = -\nabla g(x^*(T)).
\end{aligned}
\right.
\end{equation}
Moreover, the \textit{Pontryagin maximisation condition}
\begin{equation}
\label{eq:maximisation_FiniteDim}
H(t,x^*(t),p^*(t),u^*(t)) = \max_{v \in U} \, H(t,x^*(t),p^*(t),v), 
\end{equation}
holds along this extremal pair for $\Lcal^1$-almost every $t \in [0,T]$. Such a collection of optimal state, costate and control curves $(x^*(\cdot),p^*(\cdot),u^*(\cdot))$ is called an \textit{optimal Pontryagin triple} for $(\Ppazo_{\textnormal{oc}})$. Let it be noted that, since the end-points of $(\Ppazo_{\textnormal{oc}})$ are free, there are no abnormal curves stemming from the maximum principle.

\begin{lem}[Compactness and regularity of the costate]
\label{lem:CostateBound}
Let $K \subset \R^d$ be a compact set given by Lemma \ref{lem:Compactness_FiniteDim} and suppose that hypotheses \ref{hyp:Hoc} hold. Then, there exists a compact set $K' \subset \R^d$ such that $p^*(\cdot) \in \Lip([0,T],K')$. 
\end{lem}

\begin{proof}
The backward Cauchy problem satisfied by $p^*(\cdot)$ in \eqref{eq:Hamiltonian_FiniteDim} can be written explicitly as
\begin{equation*}
\dot p^*(t) = - \D_x f(t,x^*(t))^{\top} p^*(t), \qquad p^*(T) = -\nabla g(x^*(T)), 
\end{equation*}
for $\Lcal^1$-almost every $t \in [0,T]$. Since $(t,x) \in [0,T] \times K \mapsto |\D_x f(t,x)| \in \R_+$ is uniformly bounded by \ref{hyp:Hoc}-$(iii)$, it follows from Gr\"onwall's Lemma that $p^*(\cdot) \in \Lip([0,T],\R^d)$. Moreover, recall that $x \in K \mapsto |\nabla g(x)| \in \R_+$ is also uniformly bounded as a consequence of \ref{hyp:Hoc}-$(iv)$, thus upon invoking Gr\"onwall's Lemma again, there exists a compact set $K' \subset \R^d$ such that $p^*(\cdot) \in \Lip([0,T],K')$. 
\end{proof}

From now on, we denote by $\K = [0,T] \times K \times K' \times U$ the uniform compact set containing the admissible times, states, costates and controls for $(\Ppazo_{\textnormal{oc}})$, and by $\Lpazo_{\K} > 0$ be the Lipschitz constant over $\K$ of the maps $H(\cdot,\cdot,\cdot,\cdot)$, $l(\cdot,\cdot)$, $\psi(\cdot)$ and $g(\cdot)$ and and of their derivatives with respect to the variables $(x,u)$ up to the second order. Observe that both quantities exist as a consequence of Lemma \ref{lem:Compactness_FiniteDim}, Lemma \ref{lem:CostateBound} and hypotheses \ref{hyp:Hoc}.

\begin{Def}[Coercivity estimate]
\label{def:UniformCoercivity}
We say that an optimal Pontryagin triple $(x^*(\cdot),p^*(\cdot),u^*(\cdot))$ for $(\Ppazo_{\textnormal{oc}})$ satisfies the \textnormal{uniform coercivity estimate} with constant $\rho > 0$ if the following inequality holds
\begin{equation}
\label{eq:Coercivity}
\begin{aligned}
\left\langle \nabla^2_x \, g(x^*(T)) y(T) , y(T) \right\rangle & - \INTSeg{ \left\langle \nabla^2_x \, H(t,x^*(t),p^*(t),u^*(t)) y(t) , y(t) \right\rangle}{t}{0}{T} \\
& - \INTSeg{\left\langle \nabla^2_u \, H(t,x^*(t),p^*(t),u^*(t)) w(t) , w(t) \right\rangle}{t}{0}{T} \, \geq \, \rho \INTSeg{|w(t)|^2}{t}{0}{T}&
\end{aligned}
\end{equation}
for any pair of maps $(y(\cdot),w(\cdot)) \in W^{1,2}([0,T],\R^d) \times L^2([0,T],\R^d)$ solution of the \textnormal{linearised system} 
\begin{equation}
\label{eq:linearised_FiniteDim}
\left\{
\begin{aligned}
& \dot y(t) = \D_x f(t,x^*(t))y(t) + w(t), \\
& y(0) = 0 ~~ \text{and} ~~ u^*(t) + w(t) \in U ~\text{for $\Lcal^1$-almost every $t \in [0,T]$}.
\end{aligned}
\right.
\end{equation}
\end{Def}

We are now ready to recall the main contribution of \cite[Theorem 5.2]{Dontchev2019}, which we will use in the proof of Theorem \ref{thm:MainResult2}. Below, we use the notations $\Graph(x(\cdot)) := \{ (t,x(t)) ~\text{s.t.}~ t \in [0,T] \}$ and $B(x,r) \subset \R^d$ for the closed-ball of center $x \in \R^m$ and radius $r>0$ in $\R^d$.

\begin{thm}[Existence of locally optimal feedbacks for $(\Ppazo_{\textnormal{oc}})$]
\label{thm:LipFeedback}
Let $(x^*(\cdot),p^*(\cdot),$ $u^*(\cdot)) \in \Lip([0,T],K) \times \Lip([0,T],K') \times \U$ be an optimal Pontryagin triple for problem $(\Ppazo_{\textnormal{oc}})$. Suppose that  \textnormal{\textbf{(H$_{\text{oc}}$)}} hold and that $(x^*(\cdot),p^*(\cdot),u^*(\cdot))$ satisfies the uniform coercivity estimate \eqref{eq:Coercivity}-\eqref{eq:linearised_FiniteDim} with constant $\rho > 0$. 

Then, there exist positive constants $\epsilon,\eta >0$, an open subset $\Npazo \subset [0,T] \times \R^d$ and a \textnormal{locally optimal feedback} $\bar{u}(\cdot,\cdot) \in \Lip(\Npazo,\R^d)$ whose Lipschitz constant depends only on $\Lpazo_K$ and $\rho$, such that the following holds.
\begin{enumerate}
\item[\textnormal{(a)}] $\bar{u}(t,x^*(t)) = u^*(t)$ for all times $t \in [0,T]$.
\item[\textnormal{(b)}] $\big(\Graph(x^*(\cdot)) + \{0\} \times B(0,\epsilon) \big) \subset \Npazo$.
\item[\textnormal{(c)}] For each $(\tau,\xi) \in \Npazo$, the equation
\begin{equation}
\label{eq:FeedbackRestriction}
\dot x(t) = f(t,x(t)) + \bar{u}(t,x(t)), \qquad x(\tau) = \xi,
\end{equation}
has a unique solution $\hat{x}_{(\tau,\xi)}(\cdot)$ such that $\Graph(\hat{x}_{(\tau,\xi)}(\cdot)) \subset \Npazo$. 
\item[\textnormal{(d)}] The map $\hat{u}_{(\tau,\xi)} : t \in [\tau,T] \mapsto \bar{u}(t,\hat{x}_{(\tau,\xi)}(t)) \in U$ satisfies
\begin{equation*}
\INTSeg{l \big(t,\hat{x}_{(\tau,\xi)}(t),\hat{u}_{(\tau,\xi)}(t) \big)}{t}{\tau}{T} + g(\hat{x}_{(\tau,\xi)}(T)) \leq \INTSeg{l \big( t,x(t),u(t) \big)}{t}{\tau}{T} + g(x(T)), 
\end{equation*} 
for any open-loop pair $(u(\cdot),x(\cdot)) \in \U \times \Lip([\tau,T],\R^d)$ for $(\Ppazo_{oc})$ such that $\NormL{u(\cdot) - \hat{u}_{(\tau,\xi)}(\cdot)}{\infty}{[\tau,T]} \leq \eta$.
\end{enumerate}
\end{thm} 

The proof of Theorem \ref{thm:LipFeedback} in \cite{Dontchev2019} is based on a general strategy elaborated in \cite{Cibulka2018}, in which several quantitative inverse function theorems are proven under hypotheses akin to \eqref{eq:Coercivity} for non-linear optimal control problems. The key point of this approach is to remark (see e.g. \cite{Dontchev1993}) that the first-order linearisation of the PMP system  \eqref{eq:Hamiltonian_FiniteDim}-\eqref{eq:maximisation_FiniteDim} corresponds to the PMP of the linearised problem
\begin{equation*}
(\Ppazo'_{oc}) ~~ \left\{
\begin{aligned}
\min_{w(\cdot) \in \U'} & \left[ \INTSeg{ \left( \tfrac{1}{2} \langle A(t) y(t),y(t) \rangle + \tfrac{1}{2} \langle B(t) w(t),w(t) \rangle \right)}{t}{0}{T} + \tfrac{1}{2}\langle C(T) y(T),y(T) \rangle\right] \\
\text{s.t.} ~ \, & \left\{
\begin{aligned}
\dot y(t) & = \D_x f(t,x^*(t))y(t) + w(t), \\
y(0) & = 0.
\end{aligned}
\right.
\end{aligned}
\right.
\end{equation*}
associated to $(\Ppazo_{oc})$, where 
\begin{equation*}
\U' = \Big\{ v \in L^2([0,T],U) ~\text{s.t.}~ u^*(t) + v(t) \in U ~ \text{for $\Lcal^1$-almost every $t \in [0,T]$} \Big\},
\end{equation*}
and
\begin{equation*}
\left\{
\begin{aligned}
A(t) & = -\nabla^2_x H(t,x^*(t),p^*(t),u^*(t)), ~~ B(t) = -\nabla^2_u H(t,x^*(t),p^*(t),u^*(t)), \\
C(T) & = \nabla^2_x g(x^*(T)). 
\end{aligned}
\right.
\end{equation*}
for all times $t \in [0,T]$. In this context,  the coercivity estimate plays the role of a strong positive-definiteness condition on the cost of $(\Ppazo'_{oc})$ along optimal trajectories, which allows to invert the corresponding optimality system with a control on the Lipschitz constant of the inverse. In the sequel, we will use the important fact that Theorem \ref{thm:LipFeedback} holds true in any finite-dimensional Hilbert space, and in particular in $((\R^d)^N,\langle \cdot,\cdot \rangle_N)$.


\section{Non-local transport equations and mean-field optimal control}
\label{section:Fornasier2014}

In this section, we recall some results concerning continuity equations and mean-feld optimal control problems. We recall in Section \ref{subsection:ContinuityEq} concepts pertaining to non-local continuity equations, and detail in Section \ref{subsection:ExistenceFornasier2014} a powerful existence result of so-called \textit{mean-field optimal controls} for problem $(\Ppazo)$, which is borrowed from \cite{Fornasier2019}.

In the sequel, we focus on the optimal control problems in Wasserstein spaces written in the general form
\begin{equation*}
(\Ppazo) ~~ \left\{
\begin{aligned}
\min_{u \in \U} & \left[ \INTSeg{ \left( L(t,\mu(t)) + \INTDom{\psi(u(t,x))}{\R^d}{\mu(t)(x)} \right)}{t}{0}{T} + \varphi(\mu(T)) \right] \\
\text{s.t.} & \left\{
\begin{aligned}
& \partial_t \mu(t) + \nabla \cdot \big( (v(t,\mu(t),\cdot) + u(t,\cdot))\mu(t) \big) = 0, \\
& \mu(0) = \mu^0.  
\end{aligned}
\right.
\end{aligned}
\right.
\end{equation*}
Here, $\mu^0 \in \Pcal_c(\R^d)$ is a fixed initial datum, and the minimisation is taken over the set of admissible controls $\U := L^{\infty}([0,T],L^1(\R^d,U;\mu(t))$ where $(\mu(\cdot),u(\cdot))$ is a trajectory-control pair. We make the following working assumption on the data of problem $(\Ppazo)$. 

\begin{taggedhyp}{\textbn{(H)}} \hfill
\label{hyp:H}
\begin{enumerate}
\item[(i)] The set of admissible control values $U \subset \R^d$ is convex and compact. 
\item[(ii)] The control cost  $u \mapsto \psi(u) \in \R$ is $C^{2,1}$-regular and strictly convex over $U$.  
\item[(iii)] The non-local velocity field $(t,x,\mu) \mapsto v(t,\mu,x) \in \R^d$ is Lipschitz with respect to $t \in [0,T]$ and continuous in the $|\cdot|\times W_2$-topology with respect to  $(x,\mu) \in \R^d \times \Pcal_c(\R^d)$. Besides, there exists $M > 0$ such that
\begin{equation*}
|v(t,\mu,x)| ~\leq~ M \Big( 1 + |x| + \left( \mathsmaller{ \INTDom{|y|}{\R^d}{\mu(y)}} \right) \Big),
\end{equation*}
for all times $t \in [0,T]$ and any $(x,\mu) \in \R^d \times \Pcal_c(\R^d)$. Moreover, there exist $l_K,L_K > 0$ such that
\begin{equation*}
 |v(t,\mu,x) - v(t,\mu,y)| \leq l_K |x-y| \qquad \text{and} \qquad |v(t,\mu,x) - v(t,\nu,x)| \leq L_K W_2(\mu,\nu),
\end{equation*}
for any $x,y \in K$ and $\mu,\nu \in \Pcal(K)$, where $K \subset \R^d$ is an arbitrary compact set. 
\item[(iv)] The map $\mu \mapsto v(t,\mu,x) \in \R^d$ is $\Cpazo^{2,1}_{\loc}$-Wasserstein regular. 
\item[(v)] The running cost $(t,\mu) \mapsto L(t,\mu) \in \R$ is Lipschitz with respect to $t \in [0,T]$ and $\Cpazo^{2,1}_{\loc}$-Wasserstein regular with respect to $\mu \in \Pcal_c(\R^d)$. 
\item[(vi)] The final cost $\mu \mapsto \varphi(\mu) \in \R$ is $\Cpazo^{2,1}_{\loc}$-Wasserstein regular.
\end{enumerate}
\end{taggedhyp}
Observe that by classical well-posedness results for non-local continuity equations (see e.g. \cite{BonnetF2021,Pedestrian}) together with known existence results in the context mean-field optimal control problems (see e.g. \cite{Fornasier2019}), it would be sufficient to have locally Lipschitz dynamics and continuous cost functionals for solutions of $(\Ppazo)$ to exist.


\subsection{Non-local transport equations in $\R^d$}
\label{subsection:ContinuityEq}

Given a time horizon $T >0$, we denote by $\lambda := \tfrac{1}{T} \Lcal^1_{\llcorner [0,T]}$ the renormalised Lebesgue measure on $[0,T]$. For any $p \geq 1$, a curve of measures $\mu(\cdot) \in C^0([0,T],\Pcal_p(\R^d))$ can be uniquely lifted to a measure $\tilde{\mu} \in \Pcal_p([0,T] \times \R^d)$ defined by disintegration as $\tilde{\mu} = \INTDom{\mu(t)}{[0,T]}{\lambda(t)}$ in the sense of Theorem \ref{thm:Disintegration}. We shall say that $\mu(\cdot) \in C^0([0,T],\Pcal_p(\R^d))$ solves a \textit{continuity equation} with initial condition $\mu^0 \in \Pcal_p(\R^d)$ driven by a Lebesgue-Borel velocity field $\wb \in L^p([0,T] \times \R^d,\R^d;\tilde{\mu})$ provided that 
\begin{equation}
\label{eq:TransportPDE}
\left\{
\begin{aligned}
& \partial_t \mu(t) + \nabla \cdot (\wb(t,\cdot)\mu(t)) = 0, \\
& \mu(0) = \mu^0.
\end{aligned}
\right.
\end{equation} 
This equation has to be understood in duality against smooth and compactly supported functions, namely 
\begin{equation}
\label{eq:TransportPDE_Dist1}
\INTSeg{\INTDom{ \Big( \partial_t \xi(t,x) + \langle \nabla_x \xi(t,x) , \wb(t,x) \rangle \Big)}{\R^d}{\mu(t)(x)}}{t}{0}{T} = 0
\end{equation}
for any $\xi \in C^{\infty}_c((0,T) \times \R^d)$. 

It has been well-known since the works of Ambrosio in \cite{AmbrosioPDE} (see also \cite[Chapter 8]{AGS}) that weak solutions of continuity equations can exist in this low regularity context. However as already explained in the introduction above, such solutions are not well tailored to the practical investigation of mean-field control problem. Thus in Theorem \ref{thm:Classical_Wellposedness} below, we recall an existence result which was first derived in \cite{Pedestrian}, and that is concerned with classical well-posedness for \textit{non-local transport} equations in $(\Pcal_c(\R^d),W_1)$ under stronger regularity assumptions.

\begin{thm}[Well-posedness of non-local transport equations]
\label{thm:Classical_Wellposedness}
Let $v : (t,\mu,x) \in [0,T] \times \Pcal_c(\R^d) \times \R^d \rightarrow \R^d$ be a non-local velocity field satisfying hypotheses \ref{hyp:H}-$(iii)$. Then for each $\mu^0 \in \Pcal_c(\R^d)$, there exists a unique solution $\mu(\cdot) \in \Lip([0,T],\Pcal_c(\R^d))$ of \eqref{eq:TransportPDE} driven by $\wb : (t,x) \in [0,T] \mapsto v(t,\mu(t),x) \in \R^d$. Furthermore, there exist constants $R_T,L_T > 0$ such that
\begin{equation*}
\supp(\mu(t)) \subset B(0,R_T) \qquad \text{and} \qquad W_1(\mu(t),\mu(s)) \leq L_T |t-s|, 
\end{equation*} 
for all times $s,t \in [0,T]$. 
\end{thm}


\subsection{Existence of mean-field optimal controls for problem $(\Ppazo)$}
\label{subsection:ExistenceFornasier2014}

In this section, we show how problem $(\Ppazo)$ can be reformulated so as to encompass a suitable sequence of approximating discrete problems $(\Ppazo_N)$. We subsequently recall a powerful existence result derived in \cite{Fornasier2019} for general multi-agent optimal control problems formulated in Wasserstein spaces.

We start by fixing an integer $N \geq 1$, an initial datum $\xb^0_N \in (\R^d)^N$, and the associated empirical measure $\mu^0_N = \mu[\xb^0_N]$ as in Section \ref{subsection:DiscreteMeasures}. As exposed in the introduction, we consider the family of discrete problems
\begin{equation*}
\begin{aligned}
(\Ppazo_N) ~~ \left\{ 
\begin{aligned}
\min_{\ub(\cdot) \in \U_N} & \left[\INTSeg{\Big( \Lb_N(t,\xb(t)) + \frac{1}{N}\sum_{i=1}^N \psi(u_i(t)) \Big)}{t}{0}{T} + \Bvarphi_N(\xb(T)) \right] \\
\text{s.t.} ~ & \left\{
\begin{aligned}
& \dot x_i(t) = \vb_N(t,\xb(t),x_i(t)) + u_i(t), \\
& x_i(0) = x_i^0, 
\end{aligned}
\right.
\end{aligned}
\right.
\end{aligned}
\end{equation*}
with $\U_N = L^{\infty}([0,T],U^N)$, and where the mean-field approximating functionals are defined by 
\begin{equation}
\label{eq:MF_ApproximatingSeq}
\vb_N(t,\xb,x) := v(t,\mu[\xb],x), \qquad \Lb_N(t,\xb) := L(t,\mu[\xb]) \qquad \text{and} \qquad \Bvarphi_N(\xb) := \varphi(\mu[\xb]), 
\end{equation}
for any $(t,\xb,x) \in [0,T] \times (\R^d)^N \times \R^d$. It can be checked that as a consequence of hypotheses \ref{hyp:H}, the problems $(\Ppazo_N)$ satisfy hypotheses \ref{hyp:Hoc}. We can thus deduce the following lemma directly from Proposition \ref{prop:Existence_FiniteDim}.

\begin{lem}[Existence of solutions for $(\Ppazo_N)$]
\label{lem:Existence_MF}
Under hypotheses \ref{hyp:H} for each $N \geq 1$, there exists an optimal trajectory-control pair $(\xb_N^*(\cdot),\ub^*_N(\cdot)) \in \Lip([0,T],(\R^d)^N) \times \U_N$ solution of $(\Ppazo_N)$. 
\end{lem}

We proceed by recasting problem $(\Ppazo)$ into a framework which also encompasses the sequence of problems $(\Ppazo_N)$. Recall that, by Definition \ref{def:RadonNikodym}, a vector-valued measure $\Bnu \in \M([0,T]\times \R^d,U)$ is absolutely continuous with respect to $\tilde{\mu}$ if and only if there exists a map $u(\cdot,\cdot) \in L^1([0,T] \times \R^d,U;\tilde{\mu})$ such that $\Bnu = u(\cdot,\cdot) \tilde{\mu}$. Moreover, the absolute continuity of $\Bnu$ with respect to $\tilde{\mu}$ implies the existence of a $\lambda$-almost unique family of measures $\{ \Bnu(t) \}_{t \in [0,T]}$ such that $\Bnu = \INTDom{\Bnu(t)}{[0,T]}{\lambda(t)}$ in the sense of Theorem \ref{thm:Disintegration}. Whence, problem $(\Ppazo)$ can be relaxed as
\begin{equation*}
(\Ppazo_{\textnormal{meas}}) ~~ \left\{
\begin{aligned}
\min_{\Bnu \in \Ucal} & \left[ \INTSeg{\Big( L(t,\mu(t)) + \Psi(\Bnu(t) \vert \mu(t)) \Big)}{t}{0}{T}  + \varphi(\mu(T)) \right] \\
\text{s.t.} ~ & \left\{
\begin{aligned}
& \partial_t \mu(t) + \nabla \cdot \big( v(t,\mu(t),\cdot) \mu(t) + \Bnu(t) \big) = 0, \\
& \mu(0) = \mu^0.  
\end{aligned}
\right.
\end{aligned}
\right.
\end{equation*}
where we introduced the set $\Ucal = \M([0,T] \times \R^d,U)$ of \textit{generalised measure controls}, and the map 
\begin{equation*}
\Psi( \, \cdot \, \vert \mu) : \Bsigma \in \M(\R^d,U) \mapsto \left\{ 
\begin{aligned}
& \INTDom{\psi \left( \derv{\Bsigma}{\mu}(x) \right)}{\R^d}{\mu(x)} ~~ & \text{if $\Bsigma \ll \mu$,} \\
& + \infty ~~ & \text{otherwise}.
\end{aligned}
\right.
\end{equation*} 
One can then associate to any optimal trajectory-control pair $(\xb^*_N(\cdot),\ub^*_N(\cdot)) \in \Lip([0,T],(\R^d)^N) \times \U_N$ for $(\Ppazo_N)$ a measure trajectory-control pair  $(\mu_N^*(\cdot),\Bnu^*_N) \in \Lip([0,T],\Pcal_N(\R^d)) \times \Ucal$, defined by 
\begin{equation}
\label{eq:DiscreteMeasureCurve}
\mu^*_N(\cdot) = \frac{1}{N} \sum_{i=1}^N \delta_{x_i^*(\cdot)}, \qquad \Bnu_N^* = \INTDom{ \left( \frac{1}{N} \sum_{i=1}^N u_i^*(t) \delta_{x_i^*(t)} \right) }{[0,T]}{\lambda(t)}. 
\end{equation}
In the following theorem, we state a condensed version of the main result of \cite{Fornasier2019}, which shows that this relaxation allows to prove the $\Gamma$-convergence of the discrete problems $(\Ppazo_N)$ towards $(\Ppazo)$.

\begin{thm}[Existence of mean-field optimal controls for $(\Ppazo)$]
\label{thm:Fornasier2014}
Let $\mu^0 \in \Pcal_c(\R^d)$ be given, $(\mu_N^0) \subset \Pcal_c(\R^d)$ be a sequence of uniformly compactly supported empirical measures associated with $(\xb_N^0) \subset (\R^d)^N$ such that $W_1(\mu^0_N,\mu^0) \longrightarrow 0$ as $N \rightarrow +\infty$, and assume that hypotheses \ref{hyp:H} hold. For any $N \geq 1$, let $(\xb^*_N(\cdot),\ub^*_N(\cdot)) \in \Lip([0,T],(\R^d)^N) \times \U_N$ be an optimal trajectory-control pair for $(\Ppazo_N)$ and $(\mu^*_N(\cdot),\Bnu^*_N) \in \Lip([0,T],\Pcal_N(\R^d)) \times \Ucal$ be the corresponding measure trajectory-control pair defined as in \eqref{eq:DiscreteMeasureCurve}. 

Then, there exists a pair $(\mu^*(\cdot),\Bnu^*) \in \Lip([0,T],\Pcal_c(\R^d)) \times \Ucal$ such that
\begin{equation*}
\max_{t \in [0,T]}W_1(\mu^*_N(t),\mu^*(t)) ~\underset{\mathsmaller{N \rightarrow +\infty}}{\longrightarrow}~ 0 \qquad \text{and} \qquad \Bnu^*_N ~\underset{\mathsmaller{N \rightarrow +\infty}}{\rightharpoonup^*}~ \Bnu^*,
\end{equation*}
along a suitable subsequence. Moreover, the classical trajectory-control pair
\begin{equation*}
\Big( \mu^*(\cdot),\tderv{\Bnu^*}{\tilde{\mu}^*}(\cdot,\cdot) \Big) \in \Lip([0,T],\Pcal_c(\R^d)) \times L^{\infty}([0,T] \times \R^d,U;\tilde{\mu}^*),
\end{equation*}
is optimal for $(\Ppazo)$, where $\tilde{\mu}^* = \INTDom{\mu^*(t)}{[0,T]}{\lambda(t)}$.  
\end{thm}

\begin{rmk}[Comparison between \ref{hyp:H} and the assumptions of \cite{Fornasier2019}]
In \cite{Fornasier2019}, it is assumed that $U \subset \R^d$ is a subspace of $\R^d$ in order to recover the $\Gamma-\limsup$ inequality in the proof of their main result Theorem 3.2. This hypothesis could be relaxed up to an additional projection argument by asking that $U$ is convex and closed. Besides, the requirements that $\psi(\cdot)$ is radial and super-linear at infinity are primarily used to recover integral bounds on the controls, which automatically hold in our context since we posit that the control set $U$ is compact.
\end{rmk}


\section{Proof of Theorem \ref{thm:MainResult1} and Theorem \ref{thm:MainResult2}}
\label{section:MainResult}

In this section, we prove the two main results of this article. We start by working with the discrete approximations $(\Ppazo_N)$ of $(\Ppazo)$ in order to prove Theorem \ref{thm:MainResult2}. We then proceed to recover Theorem \ref{thm:MainResult1} as a corollary, by formulating a sufficient condition under which \ref{hyp:CO_N} below holds.


\subsection{Mean-field coercivity estimate and proof of Theorem \ref{thm:MainResult2}}
\label{subsection:MainResult2} 

In this section, we start by proving Theorem \ref{thm:MainResult2}. We suppose that hypotheses \ref{hyp:H} of Section \ref{section:Fornasier2014} hold, along with the following additional \textit{mean-field coercivity} assumption.

\begin{taggedhypsing}{\textbn{(CO$_N$)}} 
\label{hyp:CO_N}
There exists a constant $\rho_T > 0$ such that for every mean-field optimal Pontryagin triple $(\xb^*_N(\cdot),\rb^*_N(\cdot),\ub^*_N(\cdot))$ for $(\Ppazo_N)$ defined in the sense of Proposition \ref{prop:MFPMP} below, the following coercivity estimate holds
\begin{equation*}
\begin{aligned}
\BHess{\xb} \Bvarphi_N[\xb_N^*(T)](\yb(T),\yb(T)) & - \INTSeg{\BHess{\xb} \H_N [t,\xb_N^*(t),\rb_N^*(t),\ub_N^*(t)](\yb(t) , \yb(t))}{t}{0}{T} \\
& - \INTSeg{\BHess{\ub} \H_N [t,\xb_N^*(t),\rb_N^*(t),\ub_N^*(t)](\wb(t) , \wb(t))}{t}{0}{T} \geq \rho_T \INTSeg{|\wb(t)|_N^2}{t}{0}{T}, 
\end{aligned}
\end{equation*}
along all the solutions $(\yb(\cdot),\wb(\cdot)) \in W^{1,2}([0,T],(\R^d)^N) \times L^2([0,T],U^N)$ of the linearised system
\begin{equation*}
\hspace{-0.1cm} \left\{
\begin{aligned}
& \dot y_i(t) = \D_x \vb_N(t,\xb^*_N(t),x_i^*(t)) y_i(t) + \tfrac{1}{N} \mathsmaller{\sum}_{j=1}^N \Db_{x_j} \vb_N(t,\xb_N^*(t),x_i^*(t)) y_j(t) + w_i(t), \\
& y_i(0) = 0 \hspace{1.23cm} \text{and} \hspace{1.15cm} \ub^*_N(t) + \wb(t) \in U^N ~~ \text{for $\Lcal^1$-almost every $t \in [0,T]$}.
\end{aligned}
\right.
\end{equation*}
\end{taggedhypsing}

Our argument is split into three steps. In Step 1, we write a PMP adapted to the mean-field structure of problem $(\Ppazo_N)$. We proceed by building in Step 2 a sequence of Lipschitz-in-space optimal control maps for the discrete problems $(\Ppazo_N)$ by combining Theorem \ref{thm:LipFeedback} and \ref{hyp:CO_N}. We then show in Step 3 that this sequence of control maps is compact in a suitable weak topology preserving its Lipschitz regularity in space, and that its limit point coincide with the mean-field optimal control introduced in Theorem \ref{thm:Fornasier2014}. 


\paragraph*{Step 1: Solutions of $(\Ppazo_N)$ and mean-field Pontryagin Maximum Principle.}

In this first step, we characterise and derive uniform estimates on the optimal pairs $(\xb^*_N(\cdot),\ub^*_N(\cdot))$ for $(\Ppazo_N)$. Our analysis is based on a reformulation of the PMP applied to $(\Ppazo_N)$ as a Hamiltonian flow with respect to the inner product $\langle \cdot , \cdot \rangle_N$. 

\begin{prop}[Characterisation of the solutions of $(\Ppazo_N)$]
\label{prop:MFPMP}
Let $(\xb_N^*(\cdot),\ub_N^*(\cdot)) \in \Lip([0,T],(\R^d)^N)) \times \U_N$ be an optimal trajectory-control pair for $(\Ppazo_N)$. Then, there exists a \textnormal{rescaled covector} $\rb_N^*(\cdot) \in \Lip([0,T],(\R^d)^N)$ such that $(\xb_N^*(\cdot),\rb_N^*(\cdot),\ub_N^*(\cdot))$ satisfies the \textnormal{mean-field Pontryagin Maximum Principle}
\begin{equation}
\label{eq:ModifiedPMP}
\left\{ 
\begin{aligned}
\dot \xb_N^*(t) & = \hspace{0.27cm} \BGrad{\rb} \, \H_N(t,\xb_N^*(t),\rb_N^*(t),\ub_N^*(t)), \quad \hspace{0.04cm} \xb_N^*(0) = \xb_N^0,\\
\dot \rb_N^*(t) & = -\BGrad{\xb} \, \H_N(t,\xb_N^*(t),\rb_N^*(t),\ub_N^*(t)), \quad \rb_N^*(T) = -\BGrad{\xb} \Bvarphi_N(\xb_N^*(T)), \\
\ub_N^*(t) & \in \underset{\vb \in U^N}{\textnormal{argmax}} ~ \H_N(t,\xb_N^*(t),\rb_N^*(t),\vb) \hspace{1.6cm} \text{for $\Lcal^1$-almost every $t \in [0,T]$},
\end{aligned}
\right.
\end{equation}
where the \textnormal{mean-field Hamiltonian} of the system is defined by 
\begin{equation}
\label{eq:Hamiltonien_MF}
\H_N(t,\xb,\rb,\ub) := \frac{1}{N} \sum_{i=1}^N \Big( \left\langle r_i , \vb_N(t,\xb,x_i) + u_i \right\rangle - \psi(u_i) \Big) - \Lb_N(t,\xb)
\end{equation} 
for all $(t,\xb,\rb,\ub) \in [0,T] \times (\R^d)^N \times (\R^d)^N \times U^N$. Furthermore, there exist uniform constants $R_T,L_T > 0$ which are independent of $N \geq 1$, such that 
\begin{equation}
\label{eq:ClaimMFMPMP}
\Graph \big( (\xb^*_N(\cdot),\rb^*_N(\cdot)) \big) \subset [0,T] \times B(0,R_T)^{2N} \qquad \text{and} \qquad \Lip \Big( (\xb^*_N(\cdot),\rb^*_N(\cdot)) \, ;[0,T] \Big) \leq L_T.
\end{equation}
\end{prop}

\begin{proof}
By hypothesis \ref{hyp:H}-$(i)$, there exists a constant $R_U > 0$ such that $U \subset B(0,R_U)$. Together with the definition \eqref{eq:MF_ApproximatingSeq} of the approximating sequences and \ref{hyp:H}-$(iii)$, this implies  
\begin{equation}
\label{eq:MomentumEstimate1}
|x_i^*(t)| \leq |x_i^0| + \INTSeg{ \big| \vb_N(s,\xb_N^*(s),x_i^*(s)) + u_i^*(t) \big|}{s}{0}{t} \leq |x_i^0| + \INTSeg{M \Big( 1 + |x_i^*(t)| + \tfrac{1}{N} \mathsmaller{\sum}_{j=1}^N |x_j^*(t)| \Big)}{s}{0}{t} + R_U T, 
\end{equation}
for all times $t \in [0,T]$. By summing over the indices $i \in \{ 1,\dots,N\}$ and applying Gr\"onwall's Lemma, there exists a constant $A_T > 0$ independent of $N \geq 1$ such that  
\begin{equation}
\label{eq:MomentumEstimate2}
\max_{t \in [0,T]} \tfrac{1}{N} \mathsmaller{\sum}_{i=1}^N |x_i^*(t)| \leq A_T. 
\end{equation}
Plugging \eqref{eq:MomentumEstimate2} into \eqref{eq:MomentumEstimate1} and applying Gr\"owall's Lemma yet another time, we recover the existence of two constants $R_T^1,L_T^1 > 0$ independent of $N \geq 1$ such that for every index $i \in \{1,\dots,N\}$, it holds
\begin{equation}
\label{eq:SuppLipBound_State}
\max_{t \in [0,T]} |x_i^*(t)| \leq R_T^1 \qquad \text{and} \qquad \Lip(x_i^*(\cdot) \, ; [0,T]) \leq L_T^1.  
\end{equation} 

As a consequence of the standard PMP applied to $(\Ppazo_N)$ (see for instance \cite[Theorem 22.2]{Clarke}), there exists a family of costate curves $\pb^* : t \in [0,T] \mapsto (p_1^*(t),\dots,p_N^*(t)) \in (\R^d)^N$ such that 
\begin{equation}
\label{eq:PMPSimple}
\left\{
\begin{aligned}
\dot x_i^*(t) & = \hspace{0.25cm} \nabla_{p_i} \Hpazo_N(t,\xb^*(t),\pb^*(t),\ub^*(t)), \quad  \hspace{0.04cm} x_i^*(0) = x_i^0, \\
\dot p_i^*(t) & = -\nabla_{x_i} \Hpazo_N(t,\xb^*(t),\pb^*(t),\ub^*(t)), \quad p_i^*(T) = -\nabla_{x_i} \Bvarphi_N(\xb^*(T)), \\
u_i^*(t) & \in \underset{v \in U}{\textnormal{argmax}} ~ \left[ \langle p_i^*(t) , v \rangle - \tfrac{1}{N} \psi(v) \right], 
\end{aligned} 
\right.
\end{equation}
where the classical Hamiltonian of the system is defined as 
\begin{equation*}
\Hpazo_N(t,\xb,\pb,\ub) = \sum_{i=1}^N \left\langle p_i , \vb_N(t,\xb,x_i) + u_i \right\rangle - \frac{1}{N} \sum_{i=1}^N \psi(u_i) - \Lb_N(t,\xb), 
\end{equation*}
for every $(t,\xb,\pb,\ub) \in [0,T] \times (\R^d)^N \times (\R^d)^N \times U^N$. Introducing the rescaled curves $r_i^*(\cdot) := N p_i^*(\cdot)$, one has
\begin{equation}
\label{eq:MFPMP1}
\dot x_i^*(t) = N \nabla_{r_i} \H_N(t,\xb^*(t),\rb^*(t),\ub^*(t)) = \BGrad{r_i} \H_N(t,\xb^*(t),\rb^*(t),\ub^*(t)),
\end{equation}
\begin{equation}
\label{eq:MFPMP2}
\begin{aligned}
\dot r_i^*(t) & = -N\nabla_{x_i} \H_N(t,\xb^*(t),\rb^*(t),\ub^*(t)) = -\BGrad{x_i} \H_N(t,\xb^*(t),\rb^*(t),\ub^*(t)),
\end{aligned}
\end{equation}
\begin{equation}
\label{eq:MFPMP3}
r_i^*(T) = -N \nabla_{x_i} \Bvarphi(\xb^*(T)) = -\BGrad{x_i} \Bvarphi(\xb^*(T)), 
\end{equation}
where we used the definition of the mean-field gradient $\BGrad{}(\bullet)$ given in Proposition \ref{prop:MF_Derivatives}. Moreover, in this setting, the maximisation condition in \eqref{eq:PMPSimple} can be rewritten for $\Lcal^1$-almost every $t \in [0,T]$ as
\begin{equation*}
u_i^*(t) \in \textnormal{argmax}_{v \in U} \left[ \langle r_i^*(t) , v \rangle - \psi(v) \right].
\end{equation*}
Merging this condition with \eqref{eq:MFPMP1}, \eqref{eq:MFPMP2} and \eqref{eq:MFPMP3}, we recover that $(\xb^*(\cdot),\rb^*(\cdot),\ub^*(\cdot))$ satisfies the mean-field Pontryagin Maximum Principle \eqref{eq:ModifiedPMP} associated with the mean-field Hamiltonian $\H_N(\cdot,\cdot,\cdot,\cdot)$. 

We now prove an estimate akin to \eqref{eq:SuppLipBound_State} for the costate variable $(\rb^*_N(\cdot))$. Observe that, as a consequence of the uniform bounds of \eqref{eq:SuppLipBound_State} and Proposition \ref{prop:MF_Derivatives}, it holds for all times $t \in [0,T]$ and any $i,j \in \{1,\dots,N\}$ that
\begin{equation*}
|\BGrad{x_i} \Lb_N(t,\xb^*(t))| = |\nabla_{\mu} L(t,\mu[\xb^*_N(t)])(x_i^*(t))| \leq \max_{\mu \in \Pcal(B(0,R_T^1))} \NormC{\nabla_{\mu} L(t,\mu)(\cdot)}{0}{B(0,R_T^1),\R^d},
\end{equation*}
and 
\begin{equation*}
|\D_x \vb_N(t,\xb^*(t),x_i^*(t))| + |\Db_{x_j} \vb_N(t,\xb^*(t),x_i^*(t))| \leq \max_{\mu \in \Pcal(B(0,R_T^1))} \NormC{\D_x v(t,\mu,\cdot) + \D_{\mu} v(t,\mu,\cdot)(\cdot)}{0}{B(0,R_T^1)^2,\R^{d \times d}}.
\end{equation*}
By invoking the $\Cpazo^{2,1}_{\loc}$-Wasserstein regularity assumptions \ref{hyp:H}-$(iv),(v)$ and by Gr\"onwall's Lemma, we obtain  
\begin{equation}
\label{eq:MomentumEstimate4}
\max_{t \in [0,T]} |r_i^*(t)| \leq C' \Big( T + |\BGrad{x_i} \Bvarphi_N(\xb^*_N(T))| \Big) e^{C'T}
\end{equation}
for all $i \in \{ 1,\dots,N\}$, where $C' > 0$ is independent of $N\geq 1$. Again as a consequence of Proposition \ref{prop:MF_Derivatives}, it holds
\begin{equation*}
|\BGrad{x_i} \Bvarphi(\xb_N^*(T))| = |\nabla_{\mu} \varphi(\mu[\xb^*_N(T)])(x_i^*(T))| \leq \max_{\mu \in \Pcal(B(0,R_T^1))} \NormC{\nabla_{\mu} \varphi(\mu)(\cdot)}{0}{B(0,R^1_T),\R^d}, 
\end{equation*}
which is uniformly bounded by hypothesis \ref{hyp:H}-$(vi)$, so that 
\begin{equation}
\label{eq:MomentumEstimate5}
\max_{t \in [0,T]} |r_i^*(t)| \leq R_T^2 \qquad \text{and} \qquad \Lip(r_i^*(\cdot);[0,T]) \leq L_T^2, 
\end{equation}
for all $i \in \{1,\dots,N\}$ and some uniform constants $R_T^2,L_T^2 > 0$. Thus, we have shown that there exist two constants $R_T,L_T > 0$ independent of $N \geq 1$, such that 
\begin{equation*}
\Graph \Big( (\xb^*(\cdot),\rb^*(\cdot)) \Big) \subset [0,T] \times B(0,R_T)^{2N} \qquad \text{and} \qquad \Lip \Big( (\xb^*(\cdot),\rb^*(\cdot)) \, ;[0,T] \Big) \leq L_T.
\end{equation*}
This concludes the proof of Proposition \ref{prop:MFPMP}.
\end{proof}

We end the first step of our proof by a simple corollary in which we provide a common Lipschitz constant for all the maps involved in $(\Ppazo_N)$ that is uniform with respect to $N \geq 1$. 

\begin{cor}
\label{cor:UniformLip}
Let $\K := [0,T] \times B(0,R_T)^{2N} \times U^N$ where $R_T > 0$ is defined as in Proposition \ref{prop:MFPMP}. Then, there exists a constant $\Lpazo_{\K} > 0$ such that 
\begin{equation*}
t \mapsto \H_N(t,\xb,\rb,\ub) \qquad \text{and} \qquad t \mapsto \Lb_N(t,\xb),
\end{equation*}
are bounded by $\Lpazo_{\K}$ and $\Lpazo_{\K}$-Lipschitz over $[0,T]$ uniformly with respect to $(\xb,\rb,\ub) \in B(0,R_T)^{2N} \times U^N$, and such that the $C^{2,1}_N$-norms defined in the sense of \eqref{eq:C2Norm} of the maps
\begin{equation*}
(\xb,\ub) \mapsto \H_N(t,\xb,\rb,\ub), \qquad \xb \mapsto \Lb_N(t,\xb), \qquad \ub \mapsto \tfrac{1}{N}\mathsmaller{\sum}_{i=1}^N \psi(u_i) \qquad \text{and} \qquad \xb \mapsto \Bvarphi_N(\xb),
\end{equation*}
are bounded by $\Lpazo_{\K}$ over $B(0,R_T)^N \times U^N$, uniformly with respect to $(t,\rb) \in [0,T] \times B(0,R_T)^N$. 
\end{cor}

\begin{proof}
This result follows directly from the the Lipschitz regularity \ref{hyp:H}-$(iii)$ of the velocity field and the $\Cpazo^{2,1}_{\loc}$-Wasserstein regularity hypotheses \ref{hyp:H}-$(iv),(v),(vi)$, along with the estimate \eqref{eq:C2_Correspondance} of Proposition \ref{prop:MF_Derivatives}. 
\end{proof}


\paragraph*{Step 2 : Construction of Lipschitz-in-space optimal controls for $(\Ppazo_N)$.}

In this second step, we associate to any optimal pair $(\xb_N^*(\cdot),\ub^*_N(\cdot))$ for $(\Ppazo_N)$ a mean-field optimal control map $u^*_N \in \Lip([0,T] \times \R^d,U))$, which Lipschitz constant with respect to the space variable is uniformly bounded with respect to $N \geq 1$. 

\begin{prop}[Existence of locally optimal uniformly-Lipschitz feedbacks for $(\Ppazo_N)$]
\label{prop:MFLip_Feedback}
Assume that hypotheses \ref{hyp:H} hold and let $(\xb^*_N(\cdot),\rb_N^*(\cdot),\ub^*_N(\cdot)) \in \Lip([0,T],B(0,R_T)^N) \times \U_N$ be an optimal Pontryagin triple for $(\Ppazo_N)$ in the sense of Proposition \ref{prop:MFPMP} along which the mean-field coercivity estimate \ref{hyp:CO_N} holds. 

Then, for any $N \geq 1$, there exists a Lipschitz map $u^*_N(\cdot,\cdot) \in \Lip([0,T]\times \R^d,U)$ such that 
\begin{equation*}
u_N^*(t,x_i(t)) = u^*_i(t) \qquad \text{and} \qquad \Lip(u^*_N(t,\cdot) \, ; \R^d) \leq \Lpazo_U, 
\end{equation*}
for all times $t \in [0,T]$, where $\Lpazo_U > 0$ is independent of $N \geq 1$. 
\end{prop}

\begin{proof}
Recall that first that by Corollary \ref{cor:UniformLip}, the bounded-Lipschitz norms in $t \in [0,T]$ and the $C^{2,1}_N$-norms in $(\xb,\ub) \in B(0,R_T)^N \times U^N$ of the datum of $(\Ppazo_N)$ are uniformly bounded over $\K = [0,T] \times B(0,R_T)^{2N} \times U^N$ by a constant $\Lpazo_{\K} > 0$. As mentioned in Section \ref{section:FiniteDimOC}, Theorem \ref{thm:LipFeedback} can be applied in $((\R^d)^N,\langle \cdot,\cdot \rangle_N)$ provided that \ref{hyp:CO_N} is indeed a strong positive-definiteness condition for the canonical linearised problem associated to $(\Ppazo_N)$. To verify this, consider $(\yb(\cdot),\sbold(\cdot),\wb(\cdot)) \in W^{1,2}([0,T],(\R^d)^N) \times W^{1,2}([0,T],(\R^d)^N) \times L^2(([0,T],U^N)$ such that $\ub_N^*(t) + \wb(t) \in U^N$ for $\Lcal^1$-almost every $t \in [0,T]$ and $i \in \{1,\dots,N\}$. Then, it holds
\begin{equation}
\label{eq:WassLin1}
\begin{aligned}
\vb_N (t,\xb_N^*(t) + \yb(t), x_i^*(t) +  y_i(t)) & = \vb_N (t , \xb_N^*(t) , x_i^*(t) \big) + \D_x \vb_N \big( t,\xb_N^*(t),x_i^*(t) ) y_i(t) \\
& \hspace{0.4cm} + \tfrac{1}{N} \mathsmaller{\sum}\limits_{j=1}^N \Db_{x_j} \vb_N ( t,\xb_N^*(t),x_i^*(t)) y_j(t) + o(|y_i(t)|) + o(|\yb(t)|_N), 
\end{aligned}
\end{equation}
for all times $t \in [0,T]$, where $\Db_{x_j} \vb_N(t,\xb,x_i)$ is the matrix whose rows are the mean-field gradients with respect to $x_j$ of the components $\xb \mapsto \vb^k_N(t,\xb,x_i)$ for $k \in \{1,\dots,d\} \in \R^d$. Analogously, one also has\footnote{Here for convenience, we use the matrix representation \eqref{eq:MF_HessianMatrix} introduced in Remark \ref{rmk:MF_Hessian} for mean-field Hessians.}
\begin{equation}
\label{eq:WassLin2}
\begin{aligned}
\BGrad{\xb} \H_N(t,& \, \xb_N^*(t)+\yb(t),\rb_N^*(t)+\sbold(t),\ub_N^*(t)+\wb(t)) \\
& = \BGrad{\xb} \H_N(t,\xb_N^*(t),\rb_N^*(t),\ub_N^*(t)) + \BHess{\xb} ~ \H_N(t,\xb_N^*(t),\rb_N^*(t),\ub_N^*(t))  \yb(t) \\
& \hspace{0.4cm} + \BHess{\rb \xb} \H_N(t,\xb_N^*(t),\rb_N^*(t),\ub_N^*(t))  \sbold(t) + o(|\yb(t)|_N) + o(|\wb(t)|_N),
\end{aligned}
\end{equation}
\begin{equation}
\label{eq:WassLin3}
\begin{aligned}
\BGrad{\ub} \H_N(t,& \, \xb_N^*(t)+\yb(t),\rb_N^*(t)+\sbold(t),\ub_N^*(t)+\wb(t)) \\
& = \BGrad{\ub} \H_N(t,\xb_N^*(t),\rb_N^*(t),\ub_N^*(t))+ \BHess{\ub} ~ \H_N(t,\xb_N^*(t),\rb_N^*(t),\ub_N^*(t))  \wb(t) \\
& \hspace{0.4cm} + \BHess{\rb \ub} \H_N(t,\xb_N^*(t),\rb_N^*(t),\ub_N^*(t))  \sbold(t) + o(|\sbold(t)|_N) + o(|\wb(t)|_N),
\end{aligned}
\end{equation}
\begin{equation}
\label{eq:WassLin4}
\BGrad{\xb} \Bvarphi_N(\xb_N^*(T) + \yb(T)) = \BGrad{\xb} \Bvarphi(\xb_N^*(T)) + \BHess{\xb} \Bvarphi_N(\xb_N^*(T))  \yb(T) + o(|\yb(T)|_N), 
\end{equation}
as a consequence of the chain rule of Proposition \ref{prop:MF_Derivatives}. Following \cite{Dontchev1993}, it can be checked that the first-order linearisation of the optimality system \eqref{eq:ModifiedPMP} obtain by combining \eqref{eq:WassLin1}, \eqref{eq:WassLin2}, \eqref{eq:WassLin3} and \eqref{eq:WassLin4} is the PMP of 
\begin{equation*}
(\Ppazo_N') ~
\left\{
\begin{aligned}
\min_{\wb(\cdot) \in \U'_N} \hspace{-0.1cm} & \left[ \INTSeg{\left( \tfrac{1}{2} \langle \Ab(t)  \yb(t) , \yb(t) \rangle_N \hspace{-0.06cm} + \hspace{-0.06cm} \tfrac{1}{2} \langle \Bb(t)  \wb(t) , \wb(t) \rangle_N \right)}{t}{0}{T} \hspace{-0.06cm} + \hspace{-0.06cm} \tfrac{1}{2} \langle \Cb(T)  \yb(T),\yb(T) \rangle_N \right] \\
\text{s.t.} & \left\{ 
\begin{aligned}
\dot y_i(t) &  = \D_x \vb_N(t,\xb^*_N(t),x_i^*(t))y_i(t) + \tfrac{1}{N} \mathsmaller{\sum}\limits_{j=1}^N \Db_{x_j} \vb_N(t,\xb^*_N(t),x_i^*(t)) y_j(t), \\
y_i(0) & = 0, 
\end{aligned}
\right.
\end{aligned}
\right.
\end{equation*}
where the set of admissible controls is defined by 
\begin{equation*}
\U'_N = \left\{ v \in L^{\infty}([0,T],U^N) ~\text{s.t.}~ \ub^*_N(t) + \wb(t) \in U^N ~ \text{for $\Lcal^1$-almost every $t \in [0,T]$} \right\},
\end{equation*}
and the matrices defining the cost functionals write
\begin{equation*}
\left\{
\begin{aligned}
\Ab(t) & = -\BHess{\xb} \H_N(t,\xb^*_N(t),\rb^*_N(t),\ub^*_N(t)),~~ \Cb(T) =\BHess{\xb} \Bvarphi_N(\xb^*_N(T)), \\
\Bb(t) & = -\BHess{\ub} \H_N(t,\xb^*_N(t),\rb^*_N(t),\ub^*_N(t)), \\
\end{aligned}
\right. 
\end{equation*}
for all times $t \in [0,T]$. Thus, the coercivity estimate \ref{hyp:CO_N} is indeed a strong positive-definiteness condition for $(\Ppazo_N')$ expressed in terms of the differential structure of $((\R^d)^N,\langle \cdot,\cdot \rangle_N)$. Hence, by Theorem \ref{thm:LipFeedback} applied to $(\Ppazo_N)$, there exists a neighbourhood $\Ncal \subset [0,T] \times (\R^d)^N$ of $\text{Graph}(\xb^*(\cdot))$ and a locally optimal feedback 
\begin{equation}
\label{eq:ProjectedControl_Def}
\tilde{\ub}_N : \Big( [0,T] \times B(0,R_T)^N \Big) \cap \Npazo \rightarrow U^N, 
\end{equation}
such that $\tilde{\ub}_N(t,\xb^*(t)) = \ub^*_N(t)$ for all times $t \in [0,T]$ and 
\begin{equation}
\label{eq:LipschitzEst1}
\big| \tilde{\ub}_N(t,\xb) - \tilde{\ub}_N(s,\yb) \big|_N \leq \Lpazo_U' \big( |t-s| + |\xb - \yb|_N \big), 
\end{equation}
for any $(t,\xb),(s,\yb) \in \Npazo$, where $\Lpazo_U' > 0$ depends only on the structural constant $\Lpazo_{\K}$ introduced in Corollary \ref{cor:UniformLip} and on the coercivity constant $\rho_T$ exhibited in \ref{hyp:CO_N}. In particular, $\Lpazo_U'$ is independent of $N \geq 1$. 

For any $i \in \{ 1,\dots,N \}$, we can in turn associate to each agent trajectory $x^*_i(\cdot)$ the projected control map
\begin{equation*}
\tilde{u}_i : (t,x) \in \Npazo_i \mapsto \tilde{\ub}_N^i(t,\hat{\xb}_i^x(t)), 
\end{equation*}
for any $(t,x) \in \Npazo_i$, where we introduced the notation
\begin{equation}
\label{eq:GluedState}
\hat{\xb}_i^x(t) := (x_1^*(t),\dots,x_{i-1}^*(t),x,x_{i+1}^*(t),\dots,x_N^*(t)), 
\end{equation} 
and where the agent-based neighbourhoods $\Npazo_i \subset [0,T] \times \R^d$ are defined by 
\begin{equation*}
\Npazo_i := \Big\{ (t,x) \in [0,T] \times \R^d ~\text{s.t.}~ \hat{\xb}_i^x(t) \in \Npazo \Big\}.
\end{equation*}
These sets are well-defined and non-empty, since the projection operations onto coordinates are open mappings. Moreover, for any $t \in [0,T]$ and $x,y \in \R^d$ such that $(t,x),(t,y) \in \Npazo_i$, it holds
\begin{equation}
\label{eq:ProjectedControl_Lip}
\begin{aligned}
|\tilde{u}_i(t,x) - \tilde{u}_i(t,y)| & = \big| \tilde{\ub}_N^i(t,\hat{\xb}_i^x(t)) - \tilde{\ub}_N^i(t,\hat{\xb}_i^y(t)) \big| \\
& \leq \bigg( \, \mathsmaller{\sum}\limits_{j=1}^N |\tilde{\ub}_N^j(t,\hat{\xb}_i^x(t)) - \tilde{\ub}_N^j(t,\hat{\xb}_i^y(t))|^2 \bigg)^{1/2} \\
& = \sqrt{N} \, \big| \tilde{\ub}_N(t,\hat{\xb}_i^x(t)) - \tilde{\ub}_N(t,\hat{\xb}_i^y(t)) \big|_N ~~ \leq \sqrt{N} \Lpazo_U' |\hat{\xb}^x_i(t) - \hat{\xb}^y_i(t)|_N,
\end{aligned}
\end{equation}
as a consequence of \eqref{eq:LipschitzEst1}. Observe now that by \eqref{eq:GluedState}, the quantity $|\hat{\xb}^y_i(t) - \hat{\xb}^x_i(t)|_N$ can be further estimated as 
\begin{equation}
\label{eq:LipschitzEst2}
|\hat{\xb}^x_i(t) - \hat{\xb}^y_i(t)|_N = \bigg( \tfrac{1}{N} \mathsmaller{\sum}\limits_{j=1}^N \big| (\hat{\xb}_i^x(t))_j - (\hat{\xb}_i^y(t))_j \big|^2 \bigg)^{1/2} = ~ \tfrac{1}{\sqrt{N}} |y-x|, 
\end{equation}
for all $t \in [0,T]$, since $(\hat{\xb}_i^x(t))_j = (\hat{\xb}_i^y(t))_j = x_j^*(t)$ for any $j \neq i$. By merging \eqref{eq:ProjectedControl_Lip} and \eqref{eq:LipschitzEst2}, we recover that the maps $\tilde{u}_i(\cdot,\cdot)$ defined in \eqref{eq:ProjectedControl_Def} are $\Lpazo_U'$-Lipschitz in space over $\Npazo_i$ for any $i \in \{1,\dots,N\}$. 

To conclude the proof of Proposition \ref{prop:MFLip_Feedback}, there remains to ``patch together'' the locally optimal agent feedbacks $\tilde{u}_i(\cdot,\cdot)$ defined above. First, observe that since the maps $x \mapsto \tilde{u}_i(t,x) \in U$ are Lipschitz for any $i \in \{1,\dots,N\}$, all the individual agent trajectories are solution of the well-posed Cauchy-Lipschitz ODEs
\begin{equation*}
\dot x_i^*(t) = \vb_N(t,\xb_N^*(t),x_i^*(t)) + \tilde{u}_i(t,x_i^*(t)), 
\end{equation*}
for $\Lcal^1$-almost every $t \in [0,T]$. Besides, if $x_j^*(\tau) \in \Npazo_i$ for some time $\tau \in [0,T]$ with $j \neq i$, then the fact that $\tilde{u}_i(\cdot,\cdot)$ is a locally optimal feedback necessarily implies that $u_j^*(t) = \tilde{u}_i(t,x_j^*(t))$ for all times $t \in [\tau,T]$ such that $x_j^*(t) \in \Npazo_i$. Thus, no finite-time collisions can occur between agents, so that the sets $\Npazo_i$ can be chosen to be disjoint and the map 
\begin{equation*}
u_N^* : (t,x) \in \bigcup_{i=1}^N \overline{\Npazo_i} \mapsto \tilde{u}_i(t,x) \in U \quad \text{whenever $(t,x) \in \overline{\Npazo_i}$}, 
\end{equation*} 
is well-defined. By using McShane's Extension Theorem (see e.g. \cite[Theorem 3.1]{EvansGariepy}) combined with a projection on the convex and compact set $U \subset \R^d$, one can define a global optimal control map $u^*_N : [0,T] \times \R^d \rightarrow U$ such that $u^*_N(t,x_i^*(t)) = u_i^*(t)$ for all $t \in [0,T]$ and 
\begin{equation*}
\Lip(u_N^*(t,\cdot);\R^d) \leq \Lpazo_U,
\end{equation*}
for $\Lcal^1$-almost every $t \in [0,T]$, where the new Lipschitz constant is $\Lpazo_U := \sqrt{d} \Lpazo_U'$.
\end{proof}


\paragraph*{Step 3 : Existence of Lipschitz optimal controls for problem $(\Ppazo)$.}

In this third step, we show that the sequence of optimal maps $(u_N^*(\cdot,\cdot))$ constructed in Proposition \ref{prop:MFLip_Feedback} is compact in a suitable topology and that the limits along subsequences are optimal solutions of problem $(\Ppazo)$. 

\begin{lem}[Compactness of Lipschitz-in-space optimal maps]
\label{lem:Compactness}
Let $\Lpazo_U > 0$ be a positive constant and $\Omega \subset \R^d$ be a bounded set. Then, the set 
\begin{equation*}
\U_{\Lpazo_U} = \Big\{ u(\cdot,\cdot) \in L^2([0,T],W^{1,\infty}(\Omega,U)) ~\text{s.t.}~ \sup\limits_{t \in [0,T]} \Norm{u^*(t,\cdot)}_{W^{1,\infty}(\Omega,\R^d)} \leq  \Lpazo_U \Big\},
\end{equation*}
is compact in the weak $L^2([0,T],W^{1,p}(\Omega,\R^d))$-topology for any $p \in (1,+\infty)$. 
\end{lem}

\begin{proof}
See e.g. \cite[Theorem 2.5]{Fornasier2014}.
\end{proof}

This allows to derive the following convergence result on the sequence of controls $(u_N^*(\cdot,\cdot))$ built in Step 2. 

\begin{cor}[Convergence of Lipschitz optimal control]
\label{cor:LimitControl}
There exists a map $u^*(\cdot,\cdot) \in L^2([0,T],W^{1,\infty}(\R^d,U))$ such that the sequence of Lipschitz optimal controls $(u_N^*(\cdot,\cdot))$ defined in Proposition \ref{prop:MFLip_Feedback} converges up to a subsequence towards $u^*(\cdot,\cdot)$ in the weak $L^2([0,T],W^{1,p}(\Omega,\R^d))$-topology for any $p \in (1,+\infty)$.  
\end{cor}

\begin{proof}
This result comes from a direct application of Lemma \ref{lem:Compactness} to the sequence of optimal maps built in Proposition \ref{prop:MFLip_Feedback} up to choosing $\Omega := B(0,R_T)$ and redefining $\Lpazo_U := \max \{ L_U, \Lpazo_U \}$. 
\end{proof}

We now prove that the generalised optimal control $\Bnu^* \in \Ucal$ for problem $(\Ppazo_{\textnormal{meas}})$ is induced by the Lipschitz-in-space optimal control $u^*(\cdot,\cdot)$ defined in Corollary \ref{cor:LimitControl}. By construction, it holds for any $N \geq 1$ that
\begin{equation*}
\Bnu^*_N  :=  \INTDom{\left( \frac{1}{N} \sum_{i=1}^N u_i^*(t) \delta_{x_i^*(t)} \right)}{[0,T]}{\lambda(t)} =  \INTDom{\left( \frac{1}{N} \sum_{i=1}^N u_N^*(t,x_i^*(t)) \delta_{x_i^*(t)} \right)}{[0,T]}{\lambda(t)} = u^*_N(\cdot,\cdot) \tilde{\mu}^*_N, 
\end{equation*}
where $\Bnu^*_N \in \Ucal$ denotes the generalised empirical control measure introduced in Theorem \ref{thm:Fornasier2014}. In the following proposition, we prove that the sequence $(u_N^*(\cdot,\cdot) \tilde{\mu}_N^*)$ converges weakly-$^*$ towards $u^*(\cdot,\cdot)\tilde{\mu}^*$.
  
\begin{lem}[Convergence of generalised Lipschitz optimal controls]
\label{lem:GeneralisedControl_Convergence}
Let $R_T > 0$ be given by Proposition \ref{prop:MFPMP}, $\Omega := B(0,R_T)$ and $(\mu^*_N(\cdot)) \subset \Lip([0,T],\Pcal_1(\Omega))$ be the sequence of optimal measure curves associated with $(\Ppazo_N)$. Let $(u_N^*(\cdot,\cdot)) \subset L^2([0,T],W^{1,\infty}(\Omega,U))$ be as in Proposition \ref{prop:MFLip_Feedback} and $u^*(\cdot,\cdot)$ be one of its cluster points in the weak $L^2([0,T],W^{1,p}(\Omega,U))$-topology for some $p \in (d,+\infty)$. Then, $(\Bnu^*_N) := (u_N^*(\cdot,\cdot)\tilde{\mu}^*_N)$ converges towards $\Bnu^* = u^*(\cdot,\cdot) \tilde{\mu}^*$ in the weak-$^*$ topology of $\M([0,T] \times \Omega,U)$. 
\end{lem}

\begin{proof}
Recall first that the topological dual of the Banach space $L^2([0,T],W^{1,p}(\Omega,U))$ can be identified with $L^2([0,T],W^{-1,p'}(\Omega,U))$, where $p'$ is the conjugate exponent of $p$. Hence, the fact that $u_N(\cdot,\cdot) \rightharpoonup u(\cdot,\cdot)$ in $L^2([0,T],W^{1,p}(\Omega,U))$ as $N \rightarrow +\infty$ can be reformulated as
\begin{equation}
\label{eq:WeakConvergence_Controls1}
\INTSeg{\langle \xi(t) , u_N^*(t,\cdot) \rangle_{W^{1,p}(\Omega,U)}}{t}{0}{T} ~\underset{\mathsmaller{N \rightarrow +\infty}}{\longrightarrow}~ \INTSeg{\langle \xi(t) , u^*(t,\cdot) \rangle_{W^{1,p}(\Omega,U)}}{t}{0}{T},
\end{equation}
for any $\xi \in L^{2}([0,T],W^{-1,p'}(\Omega,\R^d))$, where $\langle \cdot,\cdot \rangle_{W^{1,p}(\Omega,U)}$ denotes the duality bracket of $W^{1,p}(\Omega,U)$.

Since we assumed that $p \in (d,+\infty)$, it holds by Morrey's Embedding (see e.g. \cite[Theorem 9.12]{Brezis}) that $W^{1,p}(\Omega,U) \subset C^0(\Omega,U)$. By taking the topological dual of this inclusion, we obtain that $\M(\Omega,U) \subset W^{-1,p'}(\Omega,U)$. This relation, combined with \eqref{eq:DualityBracket_Measures} and \eqref{eq:WeakConvergence_Controls1}, yields 
\begin{equation}
\label{eq:WeakConvergence_Controls2}
\INTSeg{\INTDom{\langle \xi(t,x) , u_N^*(t,x) \rangle}{\R^d}{\sigma(t)(x)}}{t}{0}{T} ~\underset{N \rightarrow +\infty}{\longrightarrow}~ \INTSeg{\INTDom{\langle \xi(t,x) , u^*(t,x) \rangle}{\R^d}{\sigma(t)(x)}}{t}{0}{T},
\end{equation}
for any curve $\sigma(\cdot) \in C^0([0,T],\M(\Omega,\R_+))$ and any $\xi \in C^1_c([0,T] \times \Omega,\R^d)$. Moreover, for each $N \geq 1$ it holds
\begin{equation}
\label{eq:SobolevEstimation}
\begin{aligned}
& \left| \INTSeg{\INTDom{\langle \xi(t,x) , u^*(t,x) \rangle}{\R^d}{\mu^*(t)(x)}}{t}{0}{T}  - \INTSeg{\INTDom{\langle \xi(t,x) , u_N^*(t,x) \rangle}{\R^d}{\mu_N^*(t)(x)}}{t}{0}{T}  \right| \\
\leq ~ & \left| \INTSeg{\INTDom{\langle \xi(t,x) , u^*(t,x) - u_N^*(t,x) \rangle}{\R^d}{\mu^*(t)(x)}}{t}{0}{T} \right| + \left| \INTSeg{\INTDom{\langle \xi(t,x) , u_N^*(t,x) \rangle}{\R^d}{(\mu^*(t)-\mu_N^*(t))(x)}}{t}{0}{T} \right|.
\end{aligned}
\end{equation} 
The first term in the right-hand side of \eqref{eq:SobolevEstimation} vanishes as $N \rightarrow +\infty$ as a consequence of \eqref{eq:WeakConvergence_Controls2}. By invoking Kantorovich-Rubinstein duality formula \eqref{eq:Kantorovich_duality} along with the $\Lpazo_U$-Lipschitz regularity of the maps $x \in \R^d \mapsto u_N^*(t,x) \in U$, we obtain the following upper bound on the second term in the right-hand side of \eqref{eq:SobolevEstimation}
\begin{equation*}
\left| \INTSeg{\INTDom{\langle \xi(t,x) , u_N^*(t,x) \rangle}{\R^d}{(\mu^*(t)-\mu_N^*(t))(x)}}{t}{0}{T} \right| \leq C_{\xi} \hspace{-0.05cm} \max_{t \in [0,T]} \hspace{-0.05cm} W_1(\mu_N^*(t),\mu^*(t)) ~\underset{N \rightarrow +\infty}{\longrightarrow}~ 0, 
\end{equation*}
where $C_{\xi} := \Lpazo_U \max_{t \in [0,T]} \big( \NormC{\xi(t,\cdot)}{0}{\Omega} +  \Lip(\xi(t,\cdot);\Omega) \big)$. Therefore, we recover the convergence result
\begin{equation}
\label{eq:ConvXi}
\INTSeg{\INTDom{\langle \xi(t,x) , u_N^*(t,x) \rangle}{\R^d}{\mu_N^*(t)(x)}}{t}{0}{T} ~\underset{N \rightarrow +\infty}{\longrightarrow}~ \INTSeg{\INTDom{\langle \xi(t,x) , u^*(t,x) \rangle}{\R^d}{\mu^*(t)(x)}}{t}{0}{T}, 
\end{equation}
for any $\xi \in C^1_c([0,T] \times \R^d,\R^d)$. Since the measure curves $\mu^*_N(\cdot)$ are uniformly compactly supported in $\Omega \subset \R^d$, one can show that \eqref{eq:ConvXi} holds for any $\xi \in C^0_c([0,T] \times \R^d,\R^d)$ by a classical approximation argument (see e.g. \cite{Fornasier2019}). This precisely amounts to saying that $\Bnu^*_N \rightharpoonup^* u^*(\cdot,\cdot) \tilde{\mu}^*$ as $N \rightarrow +\infty$ along the same subsequence.
\end{proof}

By uniqueness of the weak-$^*$ limit in $\M([0,T] \times \R^d,U)$, we obtain by combining Lemma \ref{lem:GeneralisedControl_Convergence} with Theorem \ref{thm:Fornasier2014} that the optimal solution $\Bnu^* \in \Ucal$ of $(\Ppazo_{\textnormal{meas}})$ is induced by $u^*(\cdot,\cdot)$. Whence the pair $(\mu^*(\cdot),u^*(\cdot,\cdot))$ is a classical optimal pair for $(\Ppazo)$, which concludes the proof of Theorem \ref{thm:MainResult2}. 


\subsection{A sufficient condition for coercivity and proof of Theorem \ref{thm:MainResult1}}
\label{subsection:MainResult1}

In this section, we prove a simple and general sufficient condition for the coercivity estimate \ref{hyp:CO_N} to hold, and use it to deduce Theorem \ref{thm:MainResult1}  from Theorem \ref{thm:MainResult2}.

\begin{prop}[A sufficient condition for mean-field coercivity]
\label{prop:SufficientCoercivity}
Let $\mu^0 \in \Pcal_c(\R^d)$ and suppose that hypotheses \ref{hyp:H} hold. Then, there exists a constant $\lambda_{(\Ppazo)} \geq 0$ such that, if the control cost $\psi : U \rightarrow \R$ is strongly convex with constant $\lambda_{\psi} > \lambda_{(\Ppazo)}$, then the coercivity \ref{hyp:CO_N} holds along any optimal mean-field Pontryagin triple $(\xb^*_N(\cdot),\rb^*_N(\cdot),\ub^*_N(\cdot))$ with $\rho_T := \lambda_{\psi} - \lambda_{(\Ppazo)}$. Moreover, the constant $\lambda_{(\Ppazo)}$ is intrinsic to $(\Ppazo)$, in the sense that it only depends on the $\Cpazo^{2}_{\loc}$-Wasserstein norms of the dynamics and cost functionals. 
\end{prop}

The main ingredient involved in this result is contained in the following technical lemma, which proof is provided for the sake of completeness. 

\begin{lem}[$C^2_N$-functions are $\lambda$-convex on products of convex compact sets]
\label{lem:LambdaConvexity}
Let $K \subset \R^d$ be a convex compact set and $\phi : \Pcal_c(\R^d) \rightarrow \R$ be $\Cpazo^{2,1}_{\loc}$-Wasserstein regular with discrete approximating sequence $(\Bphi_N(\cdot))$. Then
\begin{equation*}
\begin{aligned}
\BHess{} \Bphi_N[\xb](\hb,\hb) & \geq - \max_{\mu \in \Pcal(K)} \Big( \NormC{\D_x \nabla_{\mu} \phi(\mu)(\cdot)}{0}{K,\R^{d \times d}} + \NormC{\D^2_{\mu} \phi(\mu)(\cdot,\cdot)}{0}{K \times K,\R^{d \times d}}  \Big) |\hb|_N^2,
\end{aligned}
\end{equation*} 
for any $\xb,\hb \in K^N$. 
\end{lem}

\begin{proof}
Let $\xb,\yb \in K^N$ and $t \in [0,1]$. As a consequence of Proposition \ref{prop:MF_Derivatives}, one can write the following integral Taylor formulas for $\Bphi_N(\cdot)$
\begin{equation}
\label{eq:Lambda_Convexity1}
\left\{
\begin{aligned}
\Bphi_N((1-t)\xb + t \yb) & = \Bphi_N(\xb) + \INTSeg{\langle \BGrad{} \Bphi_N(\xb + st(\yb-\xb)) , t(\yb - \xb) \rangle_N}{s}{0}{1}, \\
\Bphi_N(\yb) & = \Bphi_N(\xb) + \INTSeg{\langle \BGrad{} \Bphi_N(\xb + s(\yb-\xb)) , \yb - \xb \rangle_N}{s}{0}{1}. 
\end{aligned}
\right.
\end{equation}
Combining the two equations of \eqref{eq:Lambda_Convexity1}, it then holds
\begin{equation}
\label{eq:LambdaConv_Exp}
\begin{aligned}
\Bphi_N((1-t)\xb + t \yb) - (1-t)\Bphi_N(\xb) - t \Bphi_N(\yb) & = \INTSeg{\big\langle \BGrad{} \Bphi_N(\xb + st(\yb-\xb)) , t(\yb-\xb) \big\rangle_N}{s}{0}{1} \\
& \hspace{0.4cm} - \INTSeg{\big\langle \BGrad{} \Bphi_N(\xb + s(\yb-\xb)) , t(\yb-\xb) \big\rangle_N}{s}{0}{1} \\
& \leq \frac{t(1-t)}{2} \Lip \big( \BGrad{} \Bphi_N(\cdot);K^N \big) |\yb - \xb|_N^2,
\end{aligned}
\end{equation}
where we used the fact that both $(\xb + st(\yb-\xb))$ and $(\xb + s(\yb-\xb))$ belong to $K^N$, since this set is convex. Therefore, we have shown that the map $\Bphi_N(\cdot)$ is $\lambda$-convex over $K^N$ with $\lambda = -\Lip(\BGrad{}\Bphi_N(\cdot);K^N)$. 

Choosing in particular $\yb = \xb + s \hb$ with $s \in (0,1)$ small, the $\lambda$-convexity \eqref{eq:LambdaConv_Exp} of $\Bphi_N(\cdot)$ can be expressed as 
\begin{equation}
\label{eq:Lambda_Convexity2}
\Bphi_N(\xb + st \hb) \leq t \big( \Bphi_N(\xb + s\hb) - \Bphi_N(\xb) \big) + s^2 \frac{t(1-t)}{2} \Lip \big( \BGrad{} \Bphi_N(\cdot);K^N \big) |\hb|_N^2. 
\end{equation}
By applying the chain rule \eqref{eq:MF_Taylor} to \eqref{eq:Lambda_Convexity2}, we obtain
\begin{equation*}
(st)^2 \BHess{} \Bphi_N[\xb](\hb,\hb) \leq s^2 t \, \BHess{} \Bphi_N[\xb](\hb,\hb) + s^2 t(1-t)\Lip \big( \BGrad{} \Bphi_N(\cdot);K^N \big) |\hb|_N^2 + o((st)^2) + o(s^2t),  
\end{equation*}
so that dividing by $s^2t > 0$ and letting $s,t \rightarrow 0^+$, we finally recover that
\begin{equation*}
\BHess{} \Bphi_N[\xb](\hb,\hb) \geq -\Lip \big( \BGrad{} \Bphi_N(\cdot);K^N \big) |\hb|_N^2, 
\end{equation*}
for any $\xb,\hb \in K^N$. One can finally check that, as a consequence of \eqref{eq:MF_Gradient}, it holds
\begin{equation*}
\begin{aligned}
|\Lip( \BGrad{} \Bphi_N(\cdot) \, ; K^N)| & \leq \max_{\xb \in K^N} \max_{|\hb|_N = 1} |\BHess{} \Bphi_N[\xb](\hb,\hb)| \\
& \leq \max_{\mu \in \Pcal(K)} \Big( \NormC{\D_x \nabla_{\mu} \phi(\mu)(\cdot)}{0}{K,\R^{d \times d}} + \NormC{\D^2_{\mu} \phi(\mu)(\cdot,\cdot)}{0}{K \times K,\R^{d \times d}}  \Big),
\end{aligned}
\end{equation*}
which concludes the proof of our claim, since $\mu[\xb] \in \Pcal(K)$ for any $\xb \in K^N$. 
\end{proof}

\begin{proof}[Proof of Proposition \ref{prop:SufficientCoercivity}]
As a consequence of hypotheses \ref{hyp:H} together with Proposition \ref{prop:MF_Derivatives}, the partial Hamiltonian $(\xb,\ub) \in B(0,R_T)^N \times U^N \mapsto \H_N(t,\xb,\rb,\ub)$ and the final cost $\xb \in (\R^d)^N \mapsto \Bvarphi_N(\xb)$ are $C^{2,1}_N$-regular uniformly with respect to $(t,\rb) \in [0,T] \times B(0,R_T)^N$, with constants $\Lpazo_{\K} > 0$ that only depend on the $\Cpazo^{2,1}_{loc}$-Wasserstein norms of the dynamics and cost functionals, where $R_T > 0$ is given by Proposition \ref{prop:MFPMP}. 

By repeating the Gr\"onwall estimates made on the costate variables in the proof of Proposition \ref{prop:MFPMP}, one can check that the solutions $\yb(\cdot)$ of the mean-field linearised system described in \ref{hyp:CO_N} are contained in a product of compact sets $K^N \subset (\R^d)^N$. Moreover, they also satisfy the estimate 
\begin{equation*}
\max \bigg\{ |\yb(T)|_N^2 ~,~ \INTSeg{|\yb(t)|_N^2}{t}{0}{T} \bigg\} \leq C_T \INTSeg{|\wb(t)|_N^2}{t}{0}{T}, 
\end{equation*}
for a given uniform constant $C_T > 0$. Merging these facts together along with the statement of Lemma \ref{lem:LambdaConvexity}, there exists an intrinsic constant $\lambda_{(\Ppazo)} \geq 0$ independent of $N \geq 1$ such that
\begin{equation}
\label{eq:LambdaConv1}
\begin{aligned}
\BHess{\xb} \Bvarphi_N[\xb_N^*(T)](\yb(T),\yb(T))- \INTSeg{ \BHess{\xb} \H_N[t,\xb_N^*(t),\rb_N^*(t),\ub_N^*(t)](\yb(t),\yb(t))}{t}{0}{T} \geq - \lambda_{(\Ppazo)} \INTSeg{|\wb(t)|^2_N}{t}{0}{T},
\end{aligned}
\end{equation}
along any linearising pair $(\yb(\cdot),\wb(\cdot))$. Observe now that if $\psi(\cdot)$ is $\lambda_{\psi}$-strongly convex, it also holds 
\begin{equation}
\label{eq:LambdaConv2}
\begin{aligned}
- \INTSeg{\BHess{\ub} \H_N[t,\xb_N^*(t),\rb_N^*(t),\ub_N^*(t)](\wb(t),\wb(t))}{t}{0}{T} & = \INTSeg{\bigg( \frac{1}{N} \sum_{i=1}^N  \langle \nabla^2 \psi(u_i^*(t)) w_i(t),w_i(t) \rangle \bigg)}{t}{0}{T} \\
& \geq \lambda_{\psi} \INTSeg{|\wb(t)|_N^2}{t}{0}{T}, 
\end{aligned}
\end{equation}
for any map $\wb(\cdot) \in L^2([0,T],U^N)$. Combining \eqref{eq:LambdaConv1} and \eqref{eq:LambdaConv2}, we obtain the uniform coercivity-type estimate
\begin{equation*}
\begin{aligned}
\BHess{\xb} \Bvarphi_N[\xb_N^*(T)](\yb(T),\yb(T)) - \INTSeg{ \BHess{\xb} \H_N[t,\xb_N^*(t),\rb_N^*(t),\ub_N^*(t)](\yb(t),\yb(t))}{t}{0}{T} & \\
- \INTSeg{\BHess{\ub} \H_N[t,\xb_N^*(t),\rb_N^*(t),\ub_N^*(t)](\wb(t),\wb(t))}{t}{0}{T} \geq & (\lambda_{\psi} - \lambda_{(\Ppazo)}) \INTSeg{|\wb(t)|_N^2}{t}{0}{T}.
\end{aligned}
\end{equation*}
Therefore, up to choosing a control cost with strong convexity constant $\lambda_{\psi} > \lambda_{(\Ppazo)}$, the coercivity estimate \ref{hyp:CO_N} holds along any optimal mean-field Pontryagin triple with $\rho_T := \lambda_{\psi} - \lambda_{(\Ppazo)}$.
\end{proof}

We can now use this intermediate result together with Theorem \ref{thm:MainResult2} to prove Theorem \ref{thm:MainResult1}. 

\begin{proof}[Proof of Theorem \ref{thm:MainResult1}]
By Theorem \ref{thm:Fornasier2014} and under hypotheses \ref{hyp:H}, one can associate to any sequence of uniformly compactly supported measures $(\mu^0_N) \subset \Pcal_c(\R^d)$ such that $W_1(\mu^0_N,\mu^0) \rightarrow 0$ as $N \rightarrow +\infty$ a sequence of generalised trajectory-control pairs $(\mu^*_N(\cdot),\Bnu^*_N) \in \Lip([0,T],\Pcal_1(\R^d)) \times \Ucal$ which converges to an optimal pair for $(\Ppazo)$. 

Moreover, we assumed that $\psi(\cdot)$ is strongly convex with $\lambda_{\psi} > \lambda_{(\Ppazo)}$, where $\lambda_{(\Ppazo)} \geq 0$ is the intrinsic constant introduced in Proposition \ref{prop:SufficientCoercivity}. Thus, the coercivity estimate \ref{hyp:CO_N} holds along any optimal mean-field Pontryagin triple for $(\Ppazo_N)$. Thus, by Theorem \ref{thm:MainResult2} there exist a constant $\Lpazo_U > 0$ together with a mean-field optimal control $u^*(\cdot,\cdot)$ for $(\Ppazo)$ such that $x \in \R^d \mapsto u^*(t,x) \in U$ is $\Lpazo_U$-Lipschitz for $\Lcal^1$-almost every $t \in [0,T]$. 
\end{proof}


\section{Sharpness of the coercivity estimate $\ref{hyp:CO_N}$}
\label{section:Discussion}

In this section, we develop an example in which the mean-field coercivity condition \ref{hyp:CO_N} is both necessary and sufficient for the Lipschitz-in-space regularity of optimal controls. With this goal in mind, we consider the mean-field optimal control problem
\begin{equation*}
(\Ppazo^V) ~~ \left\{
\begin{aligned}
\min_{u \in \U} & \left[ \frac{\lambda}{2} \INTSeg{\INTDom{|u(t,x)|^2}{\R}{\mu(t)(x)}}{t}{0}{T} - \frac{1}{2} \INTDom{\left| x - \bar{\mu}(T) \right|^2}{\R}{\mu(T)(x)} \right] \\
\text{s.t.} & \left\{
\begin{aligned}
& \partial_t \mu(t) + \nabla \cdot (u(t,\cdot) \mu(t)) = 0, \\
& \mu(0) = \tfrac{1}{2} \mathds{1}_{[-1,1]} \Lcal^1.
\end{aligned}
\right.
\end{aligned}
\right.
\end{equation*}
In $(\Ppazo^V)$, one aims at maximising the variance at time $T > 0$ of a measure curve $\mu(\cdot)$ starting from the normalised indicator function of $[-1,1]$, while penalising the running $L^2(\mu(t))$-norm of the control . Here, the set of admissible control values is $U = [-C,C]$ for a positive constant $C > 0$, and the parameter $\lambda > 0$ is the relative weight between the final cost and the control penalisation. It can be verified straightforwardly that this problem satisfies hypotheses \ref{hyp:H} of Section \ref{section:Fornasier2014}. 

Given a sequence of empirical measures $(\mu^0_N) := (\mu[\xb_N]) \subset \Pcal_N([-1,1])$ converging in the $W_1$-metric towards $\mu^0$, we can define the family $(\Ppazo^V_N)$ of discretised multi-agent problems associated to $(\Ppazo^V)$ as
\begin{equation*}
(\Ppazo^V_N) ~~ \left\{
\begin{aligned}
\min_{\ub(\cdot) \in \U_N} & \left[ \frac{\lambda}{2N} \sum_{i=1}^N \INTSeg{\hspace{-0.15cm} u_i^2(t)}{t}{0}{T} - \frac{1}{2N} \sum_{i=1}^N |x_i(T)-\bar{\xb}(T)|^2 \right] \\
\text{s.t.} & \left\{
\begin{aligned}
\dot x_i(t) & = u_i(t), \\
 x_i(0) & = x_i^0. 
\end{aligned}
\right.
\end{aligned}
\right.
\end{equation*}
where $\bar{\xb}(\cdot) = \tfrac{1}{N} \sum_{i=1}^N x_i(\cdot)$ and $\U_N = L^{\infty}([0,T],U^N)$. As a consequence of Proposition \ref{prop:Existence_FiniteDim}, there exists for any $N \geq 1$ an optimal trajectory-control pair $(\xb^*_N(\cdot),\ub^*_N(\cdot)) \in \Lip([0,T],(\R^d)^N) \times \U_N$ solution of $(\Ppazo^V_N)$. The mean-field Hamiltonian associated with $(\Ppazo^V_N)$ is given by 
\begin{equation}
\label{eq:ExplicitExp_Hamiltonian}
\H_N : (t,\xb,\rb,\ub) \in [0,T] \times (\R^2)^N \times [-C,C]^N \mapsto \frac{1}{N}\sum_{i=1}^N \Big( \langle r_i , u_i\rangle - \tfrac{\lambda}{2}|u_i|^2 \Big). 
\end{equation}
By the mean-field PMP of Proposition \ref{prop:MFPMP}, there exists a covector $\rb^*_N(\cdot) \in \Lip([0,T],\R^N)$ such that 
\begin{equation*}
\left\{
\begin{aligned}
\dot r_i^*(t) & = - \BGrad{x_i} \H_N(t,\xb^*_N(t),\rb^*_N(t),\ub^*_N(t)) = 0, \\
r_i^*(T) & = \hspace{0.27cm} \BGrad{x_i} \Var_N(\xb^*_N(T)) = x_i^*(T) - \bar{\xb}^*(T), \\
u_i^*(t) & \in \underset{v \in U}{\textnormal{argmax}} \, [ r_i^*(t) v - \tfrac{\lambda}{2}v^2]. 
\end{aligned}
\right.
\end{equation*}
Therefore, the optimal covector $\rb_N^*(\cdot)$ is constant and uniquely determined via
\begin{equation*}
r_i^*(t) = x_i^*(T) - \bar{\xb}^*(T),
\end{equation*}
for any $i \in \{1,\dots,N\}$. As a consequence of the maximisation condition one can express the components of the optimal control $\ub^*_N(\cdot)$ explicitly as
\begin{equation}
\label{eq:ExplicitExp_OptimalControl}
u_i^*(t) = \pi_U \Big( \frac{r_i^*(t)}{\lambda} \Big) \equiv \pi_{[-C,C]} \Big( \frac{x_i^*(T) - \bar{\xb}^*(T)}{\lambda} \Big), 
\end{equation}
for all $i \in \{ 1,\dots,N \}$, where $\pi_K : \R \rightarrow U$ is the standard projection onto the closed convex set $U := [-C,C] \subset \R$. It follows directly from this expression that
\begin{equation*}
\dot{\bar{\xb}}^*(t) = \frac{1}{N} \sum_{i=1}^N u_i^*(t) = \frac{1}{N} \sum_{i=1}^N \pi_{[-C,C]} \left( \frac{x_i^*(T) - \bar{\xb}^*(T)}{\lambda} \right), 
\end{equation*}
for all times $t \in [0,T]$. In the following lemma, we derive a simple and explicit necessary and sufficient condition such that \ref{hyp:CO_N} holds for $(\Ppazo^V)$. 

\begin{lem}[Charaterisation of the coercivity condition for $(\Ppazo^V)$]
\label{lem:LambdaConv_PV}
The mean-field coercivity condition \ref{hyp:CO_N} holds for $(\Ppazo^V)$ if and only if $\lambda > T$. In this case, the optimal coercivity constant is given by $\rho_T = \lambda - T$. 
\end{lem}

\begin{proof}
We first compute the Hessians involved in the coercivity estimate. For any $\xb,\yb,\ub,\wb \in \R^N$, one has
\begin{equation*}
\BHess{\xb} \Var_N[\xb](\yb,\yb) = |\yb|_N^2 - |\bar{\yb}|^2 \leq |\yb|_N^2 \qquad \text{and} \qquad \BHess{\ub} \H_N[t,\xb,\rb,\ub](\wb,\wb) = \lambda |\wb|^2_N. 
\end{equation*}
Let $(\yb(\cdot),\wb(\cdot)) \in W^{1,2}([0,T],\R^N) \times \in L^2([0,T],U^N)$ be the solution of the linearised Cauchy problem along a given optimal pair $(\xb^*_N(\cdot),\ub^*_N(\cdot))$ for $(\Ppazo^V_N)$, which writes
\begin{equation}
\label{eq:linearised_MeanField}
\dot \yb(t) = \wb(t), \qquad \yb(0) = 0, 
\end{equation}
with $\ub^*_N(t) + \wb(t) \in U^N$. By Cauchy-Schwarz's inequality, one can estimate $|\yb(T)|_N^2$ as 
\begin{equation*}
|\yb(T)|_N^2 = \Big| \mathsmaller{\INTSeg{\wb(t)}{t}{0}{T}} \, \Big|_N^2 \leq T \INTSeg{|\wb(t)|_N^2}{t}{0}{T}, 
\end{equation*}
which allows us to recover
\begin{equation*}
\begin{aligned}
- \BHess{\xb} \Var_N[\xb^*_N(T)](\yb(T),\yb(T)) - \INTSeg{\BHess{\ub} \H_N[t,\xb^*_N(t),\rb^*_N(t),\ub^*_N(t)](\wb(t),\wb(t))}{t}{0}{T} & \\
\geq (\lambda - T) \INTSeg{|\wb(t)|_N^2}{t}{0}{T}&.
\end{aligned} 
\end{equation*}
Thus, the mean-field coercivity condition \ref{hyp:CO_N} holds whenever $\lambda > T$.

Conversely, let us choose a constant admissible control perturbation $\wb_c(\cdot) := \wb_c$ such that $\bar{\wb}_c = 0$. It is always possible to make such a choice, since by \eqref{eq:ExplicitExp_OptimalControl}, there exist at least two indices $i,j \in \{1,\dots,N\}$ such that $\sign(u_i^*(t)) = - \sign(u_j^*(t))$ for all times $t \in [0,T]$. It is then sufficient to choose $\wb_c \in [-C,C]^N$ such that 
\begin{equation*}
\left\{
\begin{aligned}
(\wb_c)_i & = - \sign(u_i) \epsilon, \hspace{0.65cm}(\wb_c)_j = -(\wb_c)_i, \\
(\wb_c)_k & = 0 ~~ \text{if $k \in \{1,\dots,N \}$ and $k \neq i,j$}, \\
\end{aligned}
\right.
\end{equation*}
where $\epsilon > 0$ is a small parameter. As a consequence of \eqref{eq:linearised_MeanField}, the corresponding state perturbation $\yb_c(\cdot)$ is such that $\bar{\yb}_c(\cdot) \equiv 0$. Moreover, it also holds that
\begin{equation*}
|\yb_c(T)|_N^2 = T^2 |\wb_c|^2_N = T \INTSeg{|\wb_c|^2_N}{t}{0}{T}. 
\end{equation*}
Therefore, we have shown that this particular linearised trajectory-control pair is such that
\begin{equation*}
\begin{aligned}
& -\BHess{} \Var_N[\xb^*_N(T)](\yb_c(T),\yb_c(T)) \\
& \hspace{0.4cm} - \INTSeg{\BHess{\ub} \H_N[t,\xb^*_N(t),\rb^*_N(t),\ub^*_N(t)](\wb_c(t),\wb_c(t))}{t}{0}{T} = (\lambda - T) \INTSeg{|\wb(t)|_N^2}{t}{0}{T}, 
\end{aligned}
\end{equation*}
so that $\rho_T = \lambda - T$ is the sharp mean-field coercivity constant of $(\Ppazo)$, and \ref{hyp:CO_N} holds if and only if $\lambda > T$. 
\end{proof}

\begin{rmk}[Connection with the sufficient coercivity conditions]
\label{rmk:ConvexLag}
In Lemma \ref{lem:LambdaConv_PV}, we have proven that the sharpest constant depending for $(\Ppazo^V)$ which may serve as a sufficient lower-bound for coercivity via Proposition \ref{prop:SufficientCoercivity} is given by $\lambda_{(\Ppazo^V)} := T$. Performing the same computations in the context of a final variance \textnormal{minimisation}, our sharp constant would be given by $\lambda_{(\Ppazo^V)} := 0$, so that \ref{hyp:CO_N} would hold for every $\lambda > 0$. 
\end{rmk}

We can now use this characterisation of the coercivity condition to show that it is itself equivalent to the uniform Lipschitz regularity in space of the optimal controls. For the sake of computational tractability, we will assume that the initial condition $\xb^0 = (x_1^0,\dots,x_N^0)$ is symmetric with respect to the origin and that $\bar{\xb}^*(\cdot) \equiv 0$.

\begin{prop}[Coercivity and regularity for $(\Ppazo^V)$]
\label{prop:MF_Coercivity}
The following assertions are equivalent.
\begin{enumerate}
\item[\textnormal{(a)}] The mean-field coercivity condition $\lambda > T$ holds. 
\item[\textnormal{(b)}] For any sequence of symmetrically distributed empirical measures $(\mu_N^0) \subset \Pcal_N([-1,1])$ converging narrowly towards $\mu^0 = \tfrac{1}{2} \mathds{1}_{[-1,1]}\Lcal^1$ with associated discrete optimal pairs $(\xb^*_N(\cdot),\ub^*_N(\cdot))$, it holds 
\begin{equation*}
|u_i^*(t)-u_j^*(t)| \leq \frac{1}{\rho_T} |x_i^*(t)-x_j^*(t)|, 
\end{equation*} 
for all times $t \in [0,T]$, where $\rho_T := \lambda - T$ is the sharp coercivity constant of $(\Ppazo^V)$.
\end{enumerate}
\end{prop}

\begin{proof}
First, suppose that (a) does not hold, i.e. $\lambda \leq T$. Since the optimal controls are constant over $[0,T]$ as a consequence of \eqref{eq:ExplicitExp_OptimalControl} and we assumed that $\bar{\xb}(T) = 0$, the total cost of $(\Ppazo^V_N)$ can be rewritten as
\begin{equation*}
\Ccal(u_1,\dots,u_N) = \frac{1}{2N} \sum_{i=1}^N \Big( T(\lambda - T) u_i^2 - 2 T x_i^0 u_i - |x_i^0|^2 \Big), 
\end{equation*}
for any $N$-tuple $\ub = (u_1,\dots,u_N) \in [-C,C]^N$. Since $\lambda \leq T$, the minimum of $\Ccal$ is achieved by taking $u_i^* = \sign(x_i^0) C$ for all $i \in \{1,\dots,N\}$. This further implies 
\begin{equation*}
|u_i^*(t) - u_j^*(t)| = \left\{
\begin{aligned}
& 0 ~~ & \text{if $\sign(x_i) = \sign(x_j)$},	\\
& 2C ~~ & \text{otherwise}, 
\end{aligned}
\right.
\end{equation*}
so that for any pair of indices such that $\sign(x_i^0) = -\sign(x_j^0)$, it holds
\begin{equation}
\label{eq:NonLip}
|u_i^*(t) - u_j^*(t)| = \frac{2C}{|x_i^0-x_j^0| + 2C t} |x_i^*(t) - x_j^*(t)|.
\end{equation}
The fact that $\mu_N \rightharpoonup^* \mu^0 = \tfrac{1}{2} \mathds{1}_{[-1,1]} \Lcal^1$ as $N \rightarrow +\infty$ implies that for all $\epsilon > 0$, there exists $N_{\epsilon} \geq 1$ such that for any $N \geq N_{\epsilon}$, there exists at least one pair of indices $i,j \in \{1,\dots,N\}$ such that $\sign(x_i^0) = -\sign(x_j^0)$ and $|x_i^0-x_j^0| \leq \epsilon$. Thus, it follows from \eqref{eq:NonLip} that (b) fails to hold for some pairs of indices, at least for small times. 

Suppose now that (a) is true, i.e. $\lambda > T$, and denote by $\rho_T := \lambda - T$ the corresponding sharp coercivity constant. Let $I_N,J_N \subset \{1,\dots,N\}$ be the sets of indices defined respectively by 
\begin{equation*}
I_N = \left\{ i \in \{1,\dots,N\} ~\text{s.t.}~ |x_i^0| \leq \rho_T C \right\}, \qquad J_N = \{1,\dots,N\} \backslash I_N.
\end{equation*}
For $N \geq 1$ sufficiently large, $I_N$ is necessarily non-empty since $\rho_T > 0$ and $(\mu^0_N)$ narrowly converges towards $\mu^0$. Then for any $i \in I_N$, one has that
\begin{equation*}
|x_i^*(T)| \leq |x_i^0| + C T \leq (\rho_T + T)C = \lambda C, 
\end{equation*}
and for any such indices, the optimal controls are given by $u_i^* = \tfrac{1}{\lambda} x_i^*(T)$. In which case, one has  
\begin{equation*}
x_i^*(T) = \frac{x^0_i}{1 - T/\lambda} \qquad \text{and} \qquad u_i^*(t) = \frac{x_i^*(t)}{\rho_T + t}, 
\end{equation*}
so that 
\begin{equation}
\label{eq:Lip1}
|u_i^*(t) - u_j^*(t)| \leq \frac{1}{\rho_T + t} |x_i^*(t)-x_j^*(t)|, 
\end{equation}
for any pair of indices $i,j \in I_N$. It can be checked reciprocally that $u_i^* = \sign(x_i^0) C$ for any $i \in J_N$, which furthermore yields by \eqref{eq:NonLip} that 
\begin{equation}
\label{eq:Lip2}
|u_i^*(t) - u_j^*(t)| \leq \left\{
\begin{aligned}
& 0 ~~ & \text{if $\sign(x_i) = \sign(x_j)$}, \\
& \frac{|x_i^*(t) - x_j^*(t)|}{\rho_T + t} ~~ & \text{otherwise}.
\end{aligned}
\right.
\end{equation}
Indeed, in this case $|x_i^0 - x_j^0| \geq 2 \rho_T C$ whenever $i,j \in J_N$ and $\sign(x_i) = -\sign(x_j)$. Suppose now that we are given a pair of indices $i,j \in \{1,\dots,N \}$ such that $i \in I_N$ and $j \in J_N$. If $\sign(x_i^0) = \sign(x_j^0)$, it holds that 
\begin{equation}
\label{eq:Lip3}
\begin{aligned}
|u_i^*(t) - u_j^*(t)| = u_j^*(t) - u_i^*(t) & = \sign(x_j^0)C - \frac{x_i^*(t)}{\rho_T + t} \\
& = \frac{x_j^*(t) C}{|x_j^*(t)|} - \frac{x_i^*(t)}{\rho_T + t} \leq \frac{x_j^*(t)-x_i^*(t)}{\rho_T} = \frac{|x_i^*(t)-x_j^*(t)|}{\rho_T}, 
\end{aligned}
\end{equation}
since $|x_j^*(t)| \geq \rho_T C$ by definition of $J_N$. Symmetrically if $\sign(x_i^0) = -\sign(x_j^0)$, one can easily show that 
\begin{equation}
\label{eq:Lip4}
|u_i^*(t) - u_j^*(t)| \leq \frac{1}{\rho_T}|x_i^*(t)-x_j^*(t)|. 
\end{equation}
By merging \eqref{eq:Lip1}, \eqref{eq:Lip2}, \eqref{eq:Lip3} and \eqref{eq:Lip4}, we conclude that (b) holds whenever $\lambda > T$.
\end{proof}   

In Proposition \ref{prop:MF_Coercivity}, we have proven that the mean-field coercivity estimate is both necessary and sufficient for the existence of a uniform Lipschitz constant for the sequence of finite-dimensional optimal controls with symmetric initial data. Since we assumed that $(\mu^0_N) \subset \Pcal_N([-1,1])$ and $\mu^0_N \rightharpoonup^* \mu^0 := \tfrac{1}{2} \mathds{1}_{[-1,1]} \Lcal^1$, the fact that the initial distribution are symmetric about the origin holds up to a small error as $N \rightarrow +\infty$. Observing in addition that for any $\mu^0_N \in \Pcal([-1,1])$ the discrete optimal trajectory-control pairs $(\xb^*_N(\cdot),\ub^*_N(\cdot)) \in \Lip([0,T],\R^N) \times \U_N$ are uniquely determined, we conclude that the mean-field coercivity condition \ref{hyp:CO_N} is necessary and sufficient in the limit for the existence of a Lipschitz-in-space optimal control for $(\Ppazo^V)$.  

\medskip

{\footnotesize 
\paragraph*{Acknowledgements:} This research was partially supported by the Padua University grant SID 2018 ``Controllability, stabilizability and infimun gaps for control systems'', prot. BIRD 187147. The first author was supported by the Archim\`ede
Labex (ANR-11-LABX-0033) and by the A*MIDEX project (ANR-11-IDEX-0001-02), funded by the ``Investissements d'Avenir'' French Government program managed by the French National Research Agency (ANR).}


{\footnotesize
\bibliographystyle{plain}
\bibliography{../ControlWassersteinBib}

\end{document}}